\documentclass[a4paper,reqno,11pt]{amsart}

\newcommand{\ds}{\displaystyle}

\newcommand{\op}{\mathcal}

\newcommand{\cdc}{,\dots,}

\newcommand{\ut}{\mathsf{T}}
\newcommand{\st}{\mathsf{t}}

\newcommand{\EG}{\widetilde{EG}} 

\usepackage{amsmath}%
\usepackage{amsthm}
\usepackage{amsfonts}%
\usepackage{amssymb}%
\usepackage{graphicx}
\usepackage{xy,amsthm,enumerate,xypic,array}  

\input{xy}
\xyoption{all}

\numberwithin{equation}{section}

\newtheorem{theorem}{Theorem}[section]
\theoremstyle{plain}

\newtheorem{assumption}[theorem]{Assumption}
\newtheorem*{nonumbertheoremA}{Theorem A}

\newtheorem{conjecture}[theorem]{Conjecture}

\newtheorem{corollary}[theorem]{Corollary}
\newtheorem{lemma}[theorem]{Lemma}
\newtheorem{proposition}[theorem]{Proposition}

\theoremstyle{definition}
\newtheorem{definition}[theorem]{Definition}

\newtheorem{remark}[theorem]{Remark}

\newtheorem{construction}[theorem]{Construction}
\allowdisplaybreaks[2]

\addtocounter{MaxMatrixCols}{2}

\begin{document}

\title{Intertwining for semi-direct product operads }
\author{Benjamin C. Ward}
\email{bward@math.su.se}


\begin{abstract} This paper shows that the semi-direct product construction for $G$-operads and the levelwise Borel construction for $G$-cooperads are intertwined by the topological operadic bar construction.  En route we give a generalization of the bar construction of M.\ Ching from reduced to certain non-reduced topological operads.  \end{abstract}
\maketitle


Let $\mathsf{B}$ denote the bar construction for operads in based spaces which, after \cite{Ching}, is known to carry the structure of a cooperad when the input is reduced.  Let $G$ be a topological group and let $-\rtimes G$ denote the semi-direct product construction \cite{Marklsdp,SW} whose output is (for non-trivial $G$) never reduced.  An action of $G$ on a reduced operad $\op{P}$ induces an action of $G$ on the cooperad  $\mathsf{B}(\op{P})$, which in turn induces the structure of a cooperad on the homotopy orbits $\mathsf{B}(\op{P})_{hG}$ modeled by the levelwise Borel construction $EG_+\wedge_G \mathsf{B}(\op{P})$.  In this paper we prove:

\begin{nonumbertheoremA} Let $\op{P}$ be a reduced $G$-operad in based topological spaces.  Then $\mathsf{B}(\op{P}\rtimes G)$ carries the structure of a cooperad for which there is a homotopy equivalence of cooperads:
	\begin{equation*}
	 EG_+\wedge_G \mathsf{B}(\op{P}) \stackrel{\sim}\longrightarrow \mathsf{B}(\op{P}\rtimes G).
	\end{equation*}
\end{nonumbertheoremA}

Bar-cobar duality for topological operads and cooperads is a derived model for Koszul duality of the associated (co)homology operads.  The purpose of Theorem A is to better understand the topological underpinnings of Koszul duality in certain examples of interest.  These examples include the homologies of little disks, framed little disks, moduli spaces of surfaces, and their compactifications.

For example, we may consider the case $\op{P}=D_d$,  the reduced and based variant of the operad of little $d$-disks.  This is an operad in $SO(d)$ spaces, and the semi-direct product is the framed little disks, denoted $fD_d$.  Then our result yields $ESO(d)_+ \wedge_{SO(d)}\mathsf{B}(D_d)\sim \mathsf{B}(fD_d)$.  We have thus reduced the study of the dual of the framed little disks to the study of (equivariant) self duality of the unframed little disks.  

This result is interesting for several reasons.  First, it gives a topological foundation to the results of \cite{DCV} who prove a rational homotopy level analog of this result in the case $d=2$: that the algebraic dual of the operad for Batalin-Vilkovisky algebras $BV\cong H_\ast(fD_2)$ is a chain model for the $S^1$-equivariant homology of the little $2$-disks.  Using this result the authors are able to give a minimal cofibrant resolution of this operad.  Theorem A is a step toward a spectra level enhancement of this result: that the (non-reduced) homotopy fixed point operad of the little disks \cite{Westerland} is derived Koszul dual to the stabilization of $fD_2$ in spectra (see Section $\ref{discsec})$.  

Second, it allows us to propose topological and spectral manifestations of the algebraic Koszul duality between the homology of moduli spaces of punctured spheres (the gravity operad) and their Deligne-Mumford compactifications (the hyper-commutative operad).  
Here, however, we hasten to add that the conclusions we can draw could be strengthened by fully developing the homotopy theory of bar-cobar duality for non-reduced operads in spaces and spectra and by establishing the equivariant self-duality of little disks in spectra.  This paper motivates these future directions as we discuss in Section $\ref{discsec}$.

Finally, we mention the recent preprint \cite{KWill} in which the authors (working with real-algebraic graph complex models for $D_2$) assert that the homotopy type of a semi-direct product should be recoverable from the equivariant topology of the original operad. Our result may be viewed as a topological manifestation of this statement.

This paper is organized as follows.  Necessary background and conventions are established in Section $\ref{background}$.  In Section $\ref{cooperad}$ we recall the cooperad structure on the bar construction of a reduced operad and give a reasonably straight forward generalization to the non-reduced case (Proposition $\ref{coopmap}$).  We furthermore show that this bar construction intertwines the left and right adjoints of inclusion from reduced to non-reduced operads.  In Section $\ref{mainthmsec}$ we prove Theorem A by first constructing an explicit level-wise homotopy equivalence of $S_n$ modules between $EG_+\wedge_G \mathsf{B}(\op{P})(n)$ and $\mathsf{B}(\op{P}\rtimes G)(n)$ and then showing it is compatible with the cooperad structure from the prior section.  Finally in Section $\ref{discsec}$ we propose a connection between these results and topological Koszul duality for moduli spaces of punctured surfaces and their Deligne-Mumford compactifications.  A brief appendix outlines our conventions on $BG$, $EG$ and associated diagrammatics.

\section{Prerequisites} \label{background}

\subsection{Trees}  A graph $\ut$ is a list $(V,F,F\to V,F\to F)$ which consists of a finite set of vertices $V$, a finite set of flags (also know as half edges) $F$, a map of adjacencies $F\to V$ (telling us which vertex a flag is adjacent to) and an involution $F\to F$ (whose orbits are defined to be the edges of the graph; denote this set Ed($\ut$)).  In particular a graph has internal edges (orbits of size two) and external edges (orbits of size one).  The valence of a vertex is the number of flags adjacent to it and we define the arity of a vertex to be one less than the valence.  

Recall that a rooted tree is a connected, genus $0$ graph along with a distinguished external edge called the root edge.  We view rooted trees as directed graphs, directed toward the root edge, and use corresponding terminology such as originating, terminating, toward, incoming, and outgoing in the usual fashion.  We call the vertex adjacent to the root edge the root vertex.

A leaf in a rooted tree is an external edge which is not the root edge.  A leaf-labeled rooted tree is a rooted tree along with a bijection from the set of leaves to $\{1\cdc n\}$, for some $n\geq 1$.  
We use the following short hand terminology in this paper:

\begin{definition}  From now on in this paper the word tree refers (with one exception; see Remark $\ref{emptytreermk}$) to an isomorphism class of leaf-labeled rooted trees having no vertices of arity $0$.  Such a tree is called an $n$-tree if it has $n$ leaves.
\end{definition}

\begin{remark}\label{emptytreermk}
The fact that we prohibit vertices of arity $0$ is a convention which is consistent with studying operads with trivial arity $0$ term.  This convention has several advantages but also one defect: we must allow the ``tree'' with one edge and no vertices.  We call this the empty tree.
\end{remark}

In this paper the distinction between stable and unstable trees will be important.

\begin{definition} Basic terminology related to stable trees.
\begin{enumerate}
\item  A vertex in a tree is called a {\bf stable vertex} if its arity is greater than or equal to $2$.
\item  A tree is called a {\bf stable tree} if all of its vertices are stable.
\item  A tree is called an {\bf unstable tree} if it is not stable.
\item  The {\bf underlying stable tree} of a tree $\ut$ is the stable tree formed by removing all unstable vertices and identifying the appropriate adjacent flags.
\item  A {\bf branch} of a tree is an edge of its underlying stable tree.
\end{enumerate}
\end{definition}

We will often write $\ut$ to denote a tree (which in general is unstable), and write $\st$ for a tree which is necessarily stable.  Notice that trees with only $1$ leaf have underlying stable tree equal to the empty tree (see Remark $\ref{emptytreermk}$), since they have no stable vertices.  We refer to such a tree as having one branch.

The figure below diagrams some of our terminology. Note the leaf labeling is suppressed.

\begin{figure}[h]
	\centering
	\includegraphics[scale=.7]{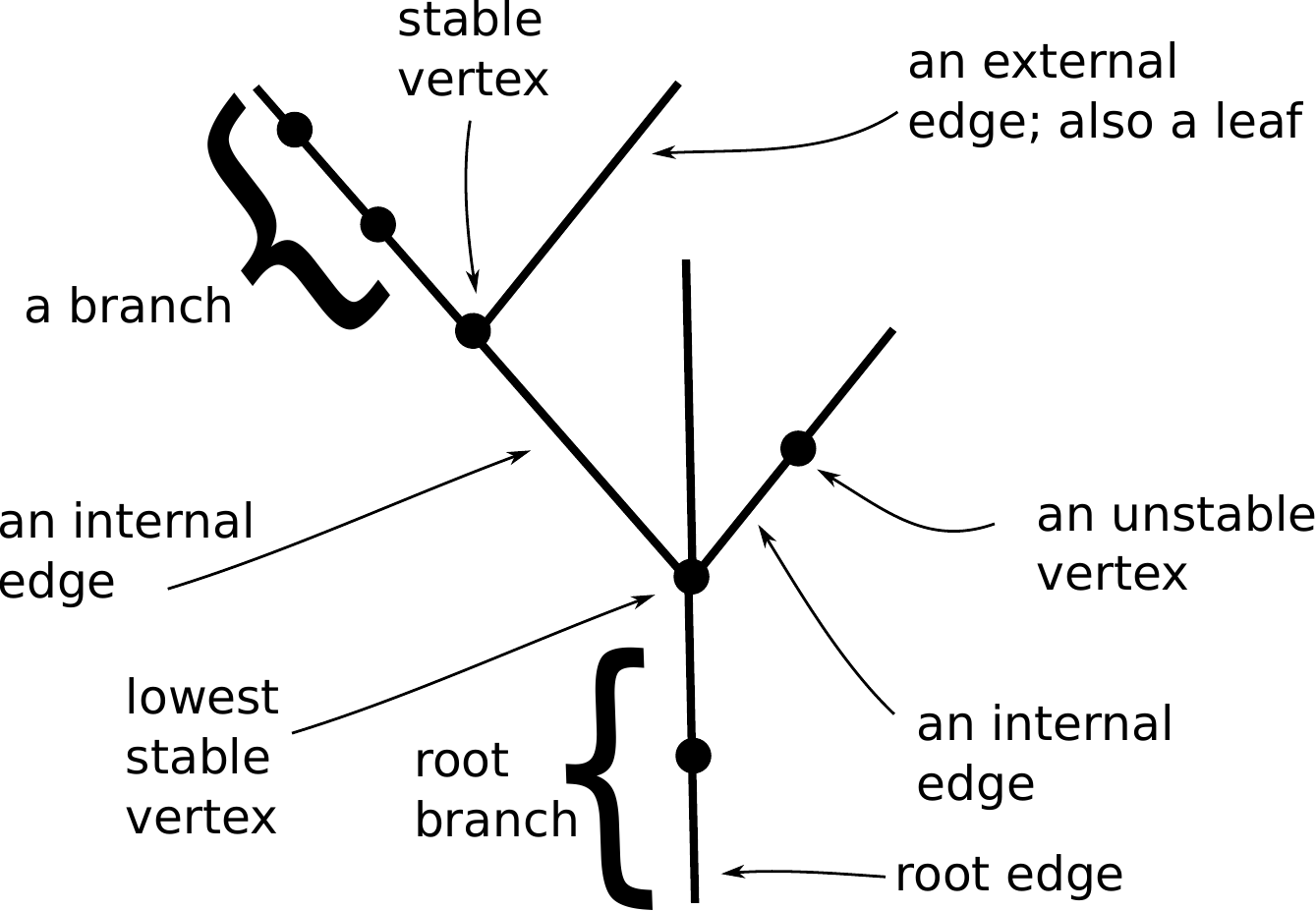}
	\label{fig:basictree}
\end{figure}

\begin{definition}
	A weighted tree is a pair $(\ut,w)$ where $\ut$ is a tree and $w\in Map(\text{Ed}(\ut),[0,1])$ such that the sum over the image of $w$ of each leaf to root path is $1$.
\end{definition}

\begin{definition}\label{wtterms}
We use the following terminology in reference to a weighted tree $(\ut,w)$.
\begin{itemize}
\item  $|E|$ is the weight of the edge $E$.
\item  The altitude, denoted $|v|$, of a vertex $v$ is the sum of the weights on edges below it.
\item  We let $v_E$ be the vertex above the edge $E$ (if applicable) and $v^E$ be the vertex below $E$ (if applicable).
\item  If $E$ is the root edge (so $v^E$ is not itself a well defined notation) we define $|v^E|:=0$.  Likewise if $E$ is a leaf edge we define $|v_E|:=1$. 
\item Let $s\in I$ (the unit interval) and let $E$ be an edge.  We say:
\begin{itemize}
	\item  $E$ is not yet active at $s$ if $s\leq |v^E|$,
	\item  $E$ is active at $s$ if $|v^E|<s<|v_E|$,
	\item  $E$ is no longer active at $s$ if $|v_E|\leq s$.
\end{itemize}
\item  We say a vertex is no longer active if the edge below it is no longer active.
\end{itemize}
\end{definition}

Associated to a weighted tree $(\ut, w)$ we have a stable weighted tree $(\st,w^\prime)$ by taking the underlying stable tree $\st$ of $\ut$ and adding the edge weights of $w$ corresponding to each branch to form the edges weights $w^\prime$.

\begin{definition}  For a tree $\ut$ and $\op{O}$ an $\mathbb{S}$-module in based spaces we define 
	\begin{equation*}
\op{O}(\ut)=\ds\bigwedge_{v\in V(\ut)}\op{O}(\text{in}(v))
	\end{equation*}
where in($v$) is the set of incoming flags at $v$, and where we extend $\op{O}$ to a functor from finite sets and bijections in the usual way (left Kan extension along the inclusion of the $\mathbb{N}$ skeleton as discussed below). 
	
	An $\op{O}$-labeling of a tree $\ut$ is a point $p_\ast\in \op{O}(\ut)$.  An $\op{O}$-labeled tree is a tree and an $\op{O}$-labeling.  An $\op{O}$-labeled, weighted tree is a tree which is both $\op{O}$-labeled and weighted.
\end{definition}

\subsection{Bar Construction}\label{barsec}
In this section we recall the simplicial bar construction in the context of operads following e.g.\ \cite{Ching}.  We work in Top$_\ast$, a nice category of based topological spaces.  We would eventually like to conclude that sequential continuity implies continuity, so nice means first-countable.

\subsubsection{Monoidal Description}  
Operads in Top$_\ast$ can be defined as monoids in a monoidal category $(\mathbb{S}$-Mod, $\op{I}$, $\circ)$.  Here $\mathbb{S}$-Mod is the category of $\mathbb{S}$-modules (also called symmetric sequences); the objects are $\mathbb{N}$ indexed sequences of $S_n$ modules and the morphisms are given level-wise.  The monoidal product is defined via:
\begin{equation*}
(X\circ Y)(n):= \ds\bigvee_{\substack{{n=r_1+...+r_m} \\ r_i\geq 1}} X(m)\wedge_{S_m}\text{Ind}_{\times_iS_{r_i}}^{S_n}(Y(r_1)\wedge...\wedge Y(r_m))
\end{equation*}
The unit $\op{I}$ is defined by $\op{I}(n)= \ast$ for $n\neq 1$ and $\op{I}(1)=S^0$.

The monoidal category $(\mathbb{S}$-Mod, $\op{I}$, $\circ)$ is monoidally equivalent to a category $(\mathbb{S}$et-Mod, $\op{I}$, $\circ)$.  Its objects are functors from the category of finite sets and bijections to Top$_\ast$ and the morphisms are level-wise.  There is a functor $\mathbb{S}$et-Mod $\to \mathbb{S}$-Mod given by restriction (viewing $n=\{1 \cdc n\}$) and this restriction functor has a left adjoint via left Kan extension.  The monoidal product on $\mathbb{S}$et-Mod can be written:

\begin{equation*}
(X\circ Y)(A):= \ds\bigvee_{A=\coprod A_j} X(J)\wedge (\wedge_{j\in J}Y(A_j))
\end{equation*}
where the $\bigvee$ is taken over all nonempty partitions of $A$.  We will move between these equivalent monoidal categories without much ado.

We say an $\mathbb{S}$-module is pointed if its arity $0$ term is $\ast$.  Likewise we say a $\mathbb{S}$et-module is pointed if it sends the empty set to $\ast$.  Let us call an operad pointed if its underlying $\mathbb{S}$-module is.

\begin{assumption}\label{pointed}
	From now on we assume all $\mathbb{S}$-modules, $\mathbb{S}et$-modules, and operads are pointed unless explicitly stated otherwise.
\end{assumption}
 
Recall that an operad $\op{Q}$ along with an operad map $\op{Q}\to \op{I}$ is called an augmented operad.  We may now defined the bar construction of an augmented operad.

\begin{definition}
	Let $\op{Q}$ be an augmented operad.  We define $\mathsf{B}(\op{Q})$ to be the $\mathbb{S}$-module given by the level-wise geometric realization of the two sided bar construction  $\op{B}_\bullet(\op{I},\op{Q},\op{I})$.  Explicitly this means;
	\begin{equation*}
	\mathsf{B}(\op{Q})(n)= \left[\ds\bigvee_{r\geq 0} \left(\op{Q}^{\circ r}(n)\wedge (\Delta_r)_+\right)\right]/\sim
	\end{equation*}
	
where $\sim$ is generated by identifying faces/cofaces (via multiplication and $\op{Q}\to \op{I}$) and degeneracies/codegeneracies (via $\op{I}\to \op{Q}$) as per usual.
\end{definition}

\subsection{Combinatorial Description of the Bar Construction}  We now give a combinatorial description of $\mathsf{B}(\op{Q})$ in terms of weighted $\op{Q}$-labeled trees.  This is basically standard and we follow \cite{Ching}, except we do not assume $\op{Q}$ is reduced (i.e. we do not assume $\op{Q}(1)= S^0$).

First consider the space of all weightings of a rooted tree $\ut$, denoted $w(\ut)$ as defined above.  Form the quotient space $\bar{w}(\ut)= w(\ut)/w_0(\ut)$, where $w_0(\ut)$ is the subspace where at least one leaf or root {\it branch} has weight $0$.  This is viewed as a based space via the quotient point.\footnote{In the case with only one leaf (so only one branch) the set $w_0(\ut)$ is empty and we take the convention (after \cite{Ching}) that quotienting by the empty set adjoins a base point, denoted $(-)_+$.}

Since $\op{Q}$ has an operadic unit, i.e.\ a map of based spaces $S^0\to \op{Q}(1)$ satisfying the usual unit axioms, we may define $e\in \op{Q}(1)$ to be the image of the point which is not the base point.

\begin{lemma}\label{description}  There is a homeomorphism of $S_n$ modules:
\begin{equation}\label{combdesc}
\mathsf{B}(\op{Q})(n)\cong \left[\bigvee_{\text{n-trees } \ut}  \op{Q}(\ut)\wedge \bar{w}(\ut)\right]/\sim
\end{equation}
where the equivalence relation is generated by the following identifications:
\begin{enumerate}
	\item A $0$ weight on an internal edge is identified with operadic composition across said edge.
	\item A label of $e$ on a vertex of arity 1 is identified with removing that vertex and adding the adjacent weights.
	\item A root or leaf edge of weight $0$, on a branch of non-zero weight, is identified with the tree formed by applying the augmentation to the adjacent unstable vertex label.
\end{enumerate}
\end{lemma}
\begin{proof} 
Recall that a tree is a level tree if it has the same number of vertices on each directed path from a leaf to the root.  The number of vertices is called the number of levels.  A weighted level tree is a level tree which is weighted such that the vertices along each root to leaf path are at the same altitudes.
	
It is standard to identify $\op{Q}^{\circ r}(n)$ with the space of $\op{Q}$-labeled $n$-trees having $r$ levels.  From this we may identify $\mathsf{B}(\op{Q})(n)$ as equivalence classes of weighted such trees.  It is also standard that, using the operadic unit, every point in the expression in line $\ref{combdesc}$ can be represented by a weighted level tree.  This is done by simply adding unary vertices labeled by $e$ as needed to level the tree.

Under this correspondence, it is straight forward to verify that items (1) and (3) in the statement correspond to the (co)face identifications and item (2) corresponds to the (co)degeneracy identifications. \end{proof}

\begin{remark} If the operad $\op{Q}$ happens to be reduced then every point in the bar construction can be represented by a stable tree.  In this case the notions of edges and branches coincide, and the content of this Lemma is manifest as Proposition 4.13 of \cite{Ching}.  In general however, our situation differs from loc.cit.\ in that it is possible to have a root or leaf edge of weight $0$ which is not identified with the base point, provided its branch is not of weight $0$.  This situation is accommodated in item 3 above.  
\end{remark}

Using this lemma, and abusing notation, we may denote points in $\mathsf{B}(\op{Q})(n)$ by $(\ut, q_\ast)$, where $\ut$ is now a weighted tree, which can be chosen to have no $0$ weight edges, and $q_\ast$ is a $\op{Q}$-labeling of $\ut$.  However, we will often view the unstable vertices as branch labels, as we now explain.  


\begin{definition}\label{branchdef}  For any point $\Psi \in \mathsf{B}(\op{Q})(n)$ we define a set br$(\Psi)$ as follows.  If $\Psi$ is the base point then br($\Psi$)$:=\emptyset$.  If $\Psi$ is not the base point then choose a representative of $\Psi$ as a weighted, $\op{Q}$-labeled tree and define br($\Psi$) to be the branches of this tree which have non-zero weight.
\end{definition}

Lemma $\ref{description}$ implies that br($\Psi$) is independent of the choice of representative: all identifications preserve the set of branches which have non-zero weight.  
We view each branch in br($\Psi$) as labeled by its set of weighted edges and its set of $\op{Q}(1)$ labeled vertices.  These labels in turn specify a point in the bar construction of the monoid $\op{Q}(1)$ (see Appendix $\ref{sec:diagrams}$), by scaling the total branch weight to one.  This total branch weight is in turn remembered by the underlying stable weighted tree, and so: 

\begin{corollary}\label{blcor} Points $\Psi\in\mathsf{B}(\op{Q})(n)$ may be represented (non-uniquely) by a list $(\st, q_\ast,\mu)$ where $(\st, q_\ast)$ is the underlying weighted stable $\op{Q}$-labeled tree and $\mu$ is map of sets $br(\Psi)\to B(\op{Q}(1))$.
\end{corollary}

\subsection{Bar Construction for $G$-operads.}\label{secbarg}

Let $G$ be a topological group.  The category of based $G$-spaces is a symmetric monoidal category under $\wedge$ with respect to the diagonal $G$ action.  By a $G$ (co)operad we mean a (co)operad in this symmetric monoidal category.  We say such a (co)operad is reduced if its underlying (co)operad is.

A key result of \cite{Ching} is the construction of a cooperad structure on $\mathsf{B}(\op{P})$, for $\op{P}$ a reduced operad.  It is easy to see that this cooperad structure is compatible with a $G$ action:

\begin{lemma}  Let $\op{P}$ be a reduced $G$-operad.  The diagonal action makes $\mathsf{B}(\op{P})$ a reduced $G$-cooperad.
\end{lemma}

\begin{proof}  By the diagonal $G$ action we mean that given a point represented as $(\st,p_\ast)\in \mathsf{B}(\op{P})(n)$, where $\st$ is a weighted tree and $p_\ast$ is a $\op{P}$-labeling of $\st$, $g\in G$ acts by preserving the weighted tree and acting diagonally on the vertices; $p_v\mapsto gp_v$ for each vertex $v\in V(\st)$.  The assumption that $\op{P}$ is a $G$-operad implies that this action is independent of the choice of representative.

To complete the proof we briefly recall the cooperad structure given in Section 4.3 of \cite{Ching} (see Section \ref{cooperad} for more detail).  It is given by de-grafting trees and then manipulating the weights so that complete edge paths still have length $1$.  In particular, the cooperad degrafting does not change the set of vertices or the vertex labels.  Since the $G$ action only acts on these vertex labels, the diagonal $G$ action and the cooperad structure commute, as desired. \end{proof}

\begin{corollary}  For a reduced $G$-operad $\op{P}$, the $\mathbb{S}$-module $EG_+\wedge_G \mathsf{B}(\op{P})$ carries a natural cooperad structure.
\end{corollary}
\begin{proof}
Consider the functor $EG_+\wedge_G-$ from based $G$-spaces to based spaces.  This functor is op-lax monoidal via the diagonal $EG\to EG\times EG$.  Since applying op-lax monoidal functors level-wise preserves cooperads, the above lemma functorially induces a cooperad structure on $EG_+\wedge_G \mathsf{B}(\op{P})$.
\end{proof}

\subsection{Semi-direct products}\label{secsdp}  The semi-direct product construction \cite{Marklsdp,SW} is a functor from unbased $G$-operads to unbased operads.  It is denoted $-\rtimes G$.  Explicitly, for an unbased $G$-operad $\op{O}$ let $(\op{O}\rtimes G)(n) := \op{O}(n)\times G^n$ with operad structure via:
\begin{equation}\label{sdp} \ \ 
(a;g_1\cdc g_n)\circ_i (b;h_1\cdc h_m):=(a\circ_i g_ib; g_1\cdc g_{i-1},g_ih_1\cdc g_ih_m, g_{i+1}\cdc g_m).
\end{equation}

There is an evident version of $-\rtimes G$ which applies to operads in based spaces.  Namely for a based $G$-operad $\op{P}$ we define $(\op{P}\rtimes G)(n) := \op{P}(n)\wedge (G_+)^{\wedge n}\cong \op{P}(n)\times G^n / \ast\times G^n $ along with the operad structure induced on such quotients by line $\ref{sdp}$.  These two constructions commute with adjoining a base point: for an unbased $G$-operad $\op{O}$ we have $(\op{O}\rtimes G)_+ \cong (\op{O}_+)\rtimes G$.

To conclude this section, consider taking the bar construction of a semi-direct product, $\mathsf{B}(\op{P}\rtimes G)$ when $\op{P}$ is a reduced operad.\footnote{using the obvious augmentation on $\op{P}\rtimes G$ which sends $G_+\to S^0$ by sending $G$ to the non-base point} If we pick a graphical representative of a non-base point, stable vertices come with a label of the form $(p, g_1\cdc g_r)\in \op{P}(r)\times G^{\times r}$.  But we can always choose a representative of the form $(p,e\cdc e)$ by decomposing the original factors of $G$ to lie above on an unstable vertex connected by a zero weight.  Thus, we can (and will) represent $\mathsf{B}(\op{P}\rtimes G)$ via trees whose stable vertices are labeled by $\op{P}\subset \op{P}\times G^{\times \ast}$.  Of course such a representation is not necessarily unique, it is subject to the semi-direct product identifications (written from now on as ``SDP''), which can be pictured as in the following figure.

\begin{figure}[h]
	\centering
	\includegraphics[scale=.35]{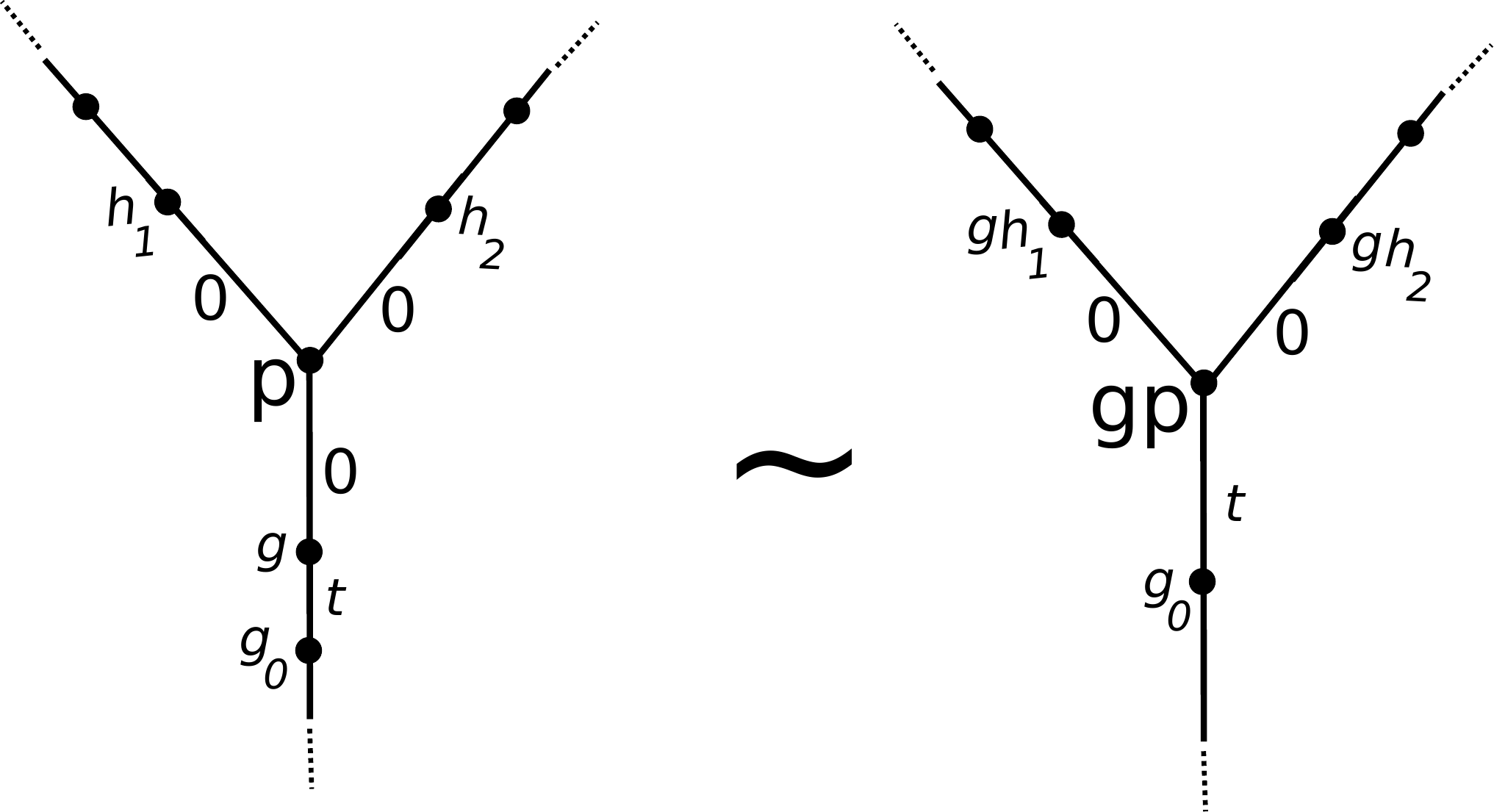}
	\label{fig:sdp}
\end{figure}

\section{Cooperad structures.}\label{cooperad}

In this section we generalize the bar construction of \cite{Ching} from reduced to certain non-reduced topological operads.  This will allows us to define a cooperad structure on the bar construction of $\op{P}\rtimes G$ which will be the subject of our main theorem in the next section.

\subsection{Recollection of the cooperad structure on reduced operads.}
We began by briefly recalling the cooperad structure on $\mathsf{B}(\op{P})$ for $\op{P}$ a reduced operad due to \cite{Ching}.  Using the definition of a cooperad as an operad in the opposite category, and in turn using the definition of an operad as a monoid in $\mathbb{S}$et-Mod (c.f.\ Subsection $\ref{barsec}$), a cooperad structure may be specified by a degrafting structure map,
\begin{equation*}
\mathsf{B}(\op{P})(A\cup_aB)\to \mathsf{B}(\op{P})(A)\wedge \mathsf{B}(\op{P})(B)
\end{equation*}
for each pair of disjoint nonempty finite sets $A,B$ and each element $a\in A$.  Here $A\cup_aB:= (A\setminus \{a\}) \cup B$.  Since the source of a structure map is a quotient of a coproduct, we may define it by fixing a stable tree $\st$ labeled by $A\cup_aB$ and defining maps
\begin{equation}\label{psm}
\op{P}(\st)\wedge \bar{w}(\st)\to \mathsf{B}(\op{P})(A)\wedge \mathsf{B}(\op{P})(B)
\end{equation}
which respect the identifications in the quotient.  This map is defined in two cases.  Either $\st$ can not be formed by grafting a $B$-labeled tree onto an $A$-labeled tree at leaf $a$ of the latter, in which case the map in line $\ref{psm}$ is defined to map to the base point.  The other case is that $\st$ can be formed by grafting a
$B$-labeled tree to an $A$-labeled tree at leaf $a$.  In this case such a grafting is unique and we may write $\st= \st_A\circ_a\st_B$ where $\st_A$ (resp.\ $\st_B$) is an $A$-labeled (resp.\ $B$-labeled) tree.  

In order to define the map in line $\ref{psm}$ it remains to say how a $\op{P}$-labeling and a weighting of $\st$ induce a $\op{P}$-labeling and a weighting of both $\st_A$ and $\st_B$.  For the $\op{P}$-labeling notice $V(\st)=V(\st_A)\sqcup V(\st_B)$ and so a $\op{P}$-labeling of $\st$ determines a canonical $\op{P}$-labeling of both $\st_A$ and $\st_B$.  Given a weighting of $\st$ we weight $\st_A$ by keeping the edge weights induced by $\st$ except at the leaf $a$, which has a unique edge weight such that the sum of root to leaf paths is $1$.  We then get a weighting of $\st_B$ by scaling all of its edges proportionally so that the sum of root to leaf paths is $1$.

It remains to check that the structure maps defined in this way are well defined and co-associative.  For this we refer to section 4.3 of \cite{Ching}.

\subsection{Cooperad structure on non-reduced operads.}  

\begin{definition}  We say an augmented topological operad $\op{Q}\to \op{I}$ is strongly augmented if the induced map $\op{Q}(1)\to S^0$ in arity $1$ sends non-base points to the non-base point.
\end{definition}

In particular if $\op{P}$ is a reduced operad, then $\op{P}\rtimes G$ is strongly augmented.  In this section we fix $\op{Q}$ to be strongly augmented and we will give a cooperad structure to $\mathsf{B}(\op{Q})$.

As above, a putative cooperad structure on $\mathsf{B}(\op{Q})$ may be defined by maps
\begin{equation}\label{psm2}
\op{Q}(\ut)\wedge \bar{w}(\ut)\to \mathsf{B}(\op{Q})(A)\wedge \mathsf{B}(\op{Q})(B)
\end{equation}
for each pair of disjoint finite nonempty sets $A,B$ and each element $a\in A$.  Let $\st$ be the underlying stable tree of $\ut$; note this is still a $A\cup_aB$ labeled tree.  We again consider two cases.  Either $\st$ can not be formed by grafting an $A$-labeled tree to a $B$-labeled tree at leaf $a$ of the former, in which case the map in line $\ref{psm2}$ is defined to map to the base point.  The other case is that $\st$ can be formed by grafting a $B$-labeled tree onto an $A$-labeled tree at leaf $a$ and we may write $\st= \st_A\circ_a\st_B$ where $\st_A$ (resp.\ $\st_B$) is a $A$-labeled (resp.\ $B$-labeled) tree.

Recall (Corollary $\ref{blcor}$) that we may depict points $\Psi\in\mathsf{B}(\op{Q})(n)$ as lists $(\st, q_\ast; \mu)$ where $(\st, q_\ast)$ is a stable tree labeled by $\op{Q}$ and $\mu$ is a map of sets $br(\Psi)\to B(\op{Q}(1))$ (called a marking of the set of branches).  Given $\Psi=(\st, q_\ast; \mu)\in \op{Q}(\ut)\wedge \bar{w}(\ut)$ we will define the structure map $\Psi\mapsto \Psi_A\wedge \Psi_B$ with $\Psi_A=(\st_A, q_\ast; \mu_A)$ and $\Psi_B=(\st_B, q_\ast; \mu_B)$.  The degrafting of stable weighted trees given above defines weightings on $\st_A$ and $\st_B$ and these come with $\op{Q}$-labelings induced as above from the bijection $V(\st)=V(\st_A)\sqcup V(\st_B)$.  This specifies the stable $\op{Q}$-labeled trees $(\st_A,q_\ast)\in \mathsf{B}(\op{Q})(A)$ and $(\st_B,q_\ast)\in \mathsf{B}(\op{Q})(B)$, and it remains to define $\mu_A$ and $\mu_B$.

There is an obvious surjective map $\text{br}(\Psi_A)\sqcup\text{br}(\Psi_B)\to\text{br}(\Psi)$ which identifies the root of $\st_B$ with the leaf labeled by $a$ of $\st_A$ and which is otherwise the natural bijective correspondence.  Composing this map with $\mu$ associates a point in $B(\op{Q}(1))$ to each point in $\text{br}(\Psi_A)\sqcup\text{br}(\Psi_B)$.  Restricting this correspondence to $\text{br}(\Psi_A)$ and $\text{br}(\Psi_B)$  defines $\mu_A$ and $\mu_B$ respectively.  This defines our putative cooperad structure.  An informal description of this cooperad structure can be given by saying we degraft the underlying stable trees, as per \cite{Ching}, and we duplicate the $\op{Q}(1)$-marking at the degrafted branch.  See Figure $\ref{fig:coop}$.

\begin{figure}[h]		
	\centering 	
	\includegraphics[scale=.4]{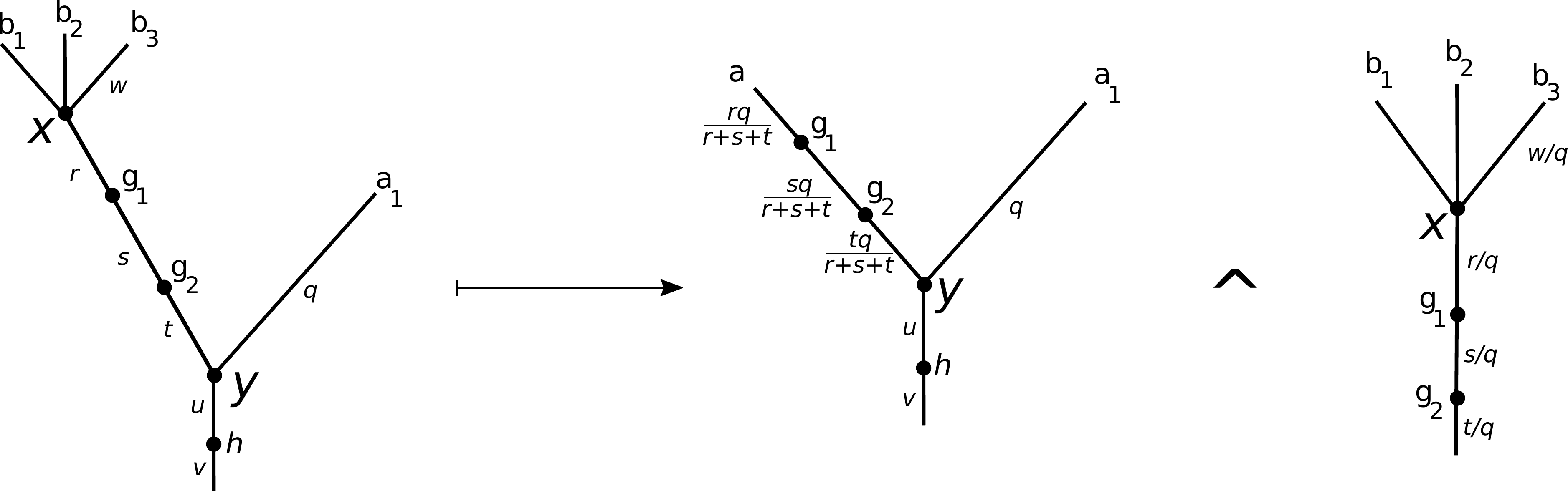}
	\caption{Cooperadic decomposition at $a$ for $A=\{a,a_1\}$, $B=\{b_1,b_2,b_3\}$.  Here $h,g_i\in \op{Q}(1)$ and  $q,r,s,t,u,v,w$ are weights. Notice that if $q$ or $r+s+t$ is $0$, the target of this decomposition is the base point.}
	\label{fig:coop}
\end{figure}

\begin{proposition}\label{coopmap}  The above structure maps induce the structure of a cooperad on $\mathsf{B}(\op{Q})$.
\end{proposition}
\begin{proof}
	We first check that the maps defined in line $\ref{psm2}$ are compatible with the identifications in the bar construction.  The argument follows as in p.871 of \cite{Ching}.  In particular we remark that if an internal branch of a representation of a point in the source has weight $0$, then degrafting before contracting maps to the base point.  On the other hand if we contract this edge, its underlying stable tree no longer is a grafting of a $B$ tree onto an $A$ tree and so maps to the base point as well by definition.
	
	The interesting new cases not covered in loc.cit.\ is if there is an edge of weight $0$ adjacent to one stable and one unstable vertex.  These correspond to the cases $r=0$ or $t=0$ in Figure $\ref{fig:coop}$ and are easily checked.  For example if $r=0$, degrafting before contacting employs (3) of Lemma $\ref{description}$ to remove the unstable vertex labeled by $g_1$ on the left hand factor while using operadic composition on the right hand factor.  This is clearly the same as first contracting and then degrafting.  Notice here we used the strongly augmented hypothesis to ensure that if $h,g_i$ are non-base points, then they are mapped to the identity by the augmentation, which has the effect of removing these vertices in the image of the decomposition.
	
	To conclude we observe the decomposition maps are coassociative.  This follows from the coassociativity in the reduced case after \cite{Ching}.  Indeed using this we need only compare the branch markings of iterated degraftings and these are given by the diagonal which is clearly coassociative. 	\end{proof}

\subsection{Relation between the reduced and non-reduced cases} \label{comparisonsec}

There is an obvious inclusion functor from the category of reduced operads to the category of (strongly augmented) operads which commutes with $\mathsf{B}$.  To distinguish between these categories we will call this functor $\iota$ in this section.  This functor has both a left and a right adjoint which we now describe.

For an augmented operad $\op{Q}$ we define a reduced operad $R(\op{Q})$ by 
\begin{equation*}
R(\op{Q})(n) = \begin{cases}  
S^0 & \text{ if } n=1 \\
\op{Q}(n) & \text{ if } n\geq 2
\end{cases}
\end{equation*}
with operad structure maps induced by either those in $\op{Q}$ when the arities are $\geq 2$ or trivial structure maps involving arity $1$.  We furthermore define $R$ of a morphism of operads to be the restricted morphism.  Then $R$ is a functor and the fact that $R$ is a right adjoint to $\iota$ may be easily verified.  Indeed $R\iota$ is the identity, while the counit of the adjunction $\iota R\Rightarrow id_{ops} $ is via the unit map $S^0=\iota R(\op{Q})(1)\to\op{Q}(1)$ and the identity in all other arrows.

For an augmented operad $\op{Q}$ we define the reduced operad $L(\op{Q})$ to be the pushout of 
\begin{equation*}
\op{Q}\leftarrow \op{Q}(1)\to \op{I}
\end{equation*}
in the category of operads.  Recall (subsection $\ref{barsec}$) that $\op{I}$ is the trivial operad having $S^0$ in arity $1$ and a base point in all other arities.  Likewise $\op{Q}(1)$ refers to the operad having $\op{Q}(1)$ in arity $1$ and a base point in all other arities.  The arrow 
$\op{Q}(1)\to \op{I}$ comes from the augmentation.  The fact that $L(\op{Q})$ is a reduced operad can be seen, for example, by the fact that restriction to operads concentrated in arity $1$ is a left adjoint of inclusion (it's also a right adjoint of inclusion) and so preserves pushouts.  Extending $L$ to a functor we observe that $L$ is a left adjoint of $\iota$.  Indeed $L\iota$ is the identity, while the unit of the adjunction $id_{ops} \Rightarrow \iota L $ is via the cocone map $\op{Q}\to L(\op{Q}) = \iota L(\op{Q})$.

We will also need cooperadic analogs of the triple of adjoint functors $(L,\iota, R)$.  For this we follow section 5 of \cite{Ching} and continue with the definition of a cooperad as an operad in Top$_\ast^{op}$ and let the cobar construction $\Omega(\op{C}):=\mathsf{B}(\op{C}^{op})^{op}$. We then define a triple of adjoint functors $(\tilde{R},\tilde{\iota},\tilde{L})$ between coaugmented cooperads and reduced cooperads as the levelwise opposites of $(L,\iota, R)$. That is, $\tilde{\iota}$ is inclusion of reduced cooperads into cooperads, $\tilde{R}$ is restriction to trivial arity $1$ and $\tilde{L}(\op{C})$ is the pullback of 
$\op{C}\to \op{C}(1)\leftarrow \op{I}$.

We conclude this section by observing that $L$ and $R$ intertwine the bar-cobar construction for non-reduced operads:

\begin{proposition}\label{redint}  There are isomorphisms of functors $\mathsf{B}R\cong \tilde{L}\mathsf{B}$ and $\Omega \tilde{R}\cong L\Omega$.
\end{proposition}
\begin{proof}  We will sketch the proof for the first statement with the second statement following similarly.  Let $\op{Q}$ be an operad and it suffices to prove that $\mathsf{B}(R\op{Q})$ is a pullback of $\mathsf{B}(\op{Q})\to \mathsf{B}(\op{Q})(1) \leftarrow \op{I}$.  First note that we have cooperad maps $\mathsf{B}(\op{Q})\leftarrow \mathsf{B}(R\op{Q}) \rightarrow \op{I}$ by taking $\mathsf{B}$ of the inclusion $R\op{Q}\to \op{Q}$ and the counit of $\mathsf{B}(R\op{Q})$.  These serve as the cone maps for the pullback, and it suffices to establish universality.
	
	Fix a cooperad $\op{C}$ along with a cone from $\op{C}$ to the diagram $\mathsf{B}(\op{Q})\to \mathsf{B}(\op{Q})(1) \leftarrow \op{I}$.  Denote the cone map $\op{C}\to \mathsf{B}(\op{Q})$ by $\eta$.  Any $\eta(c)\in \mathsf{B}(\op{Q})(n)$, may be decomposed via the cooperadic structure maps to a list of factors with no internal branches.  This list may in turn be decomposed so that the labels of branches of $\eta(c)$ appear as factors in the decomposition.  Since $\eta$ is a cooperad map, we may realize this decomposition in $\op{C}$ before applying $\eta$.  Since $\eta$ is part of a cone over the diagram, the arity $1$ factors in the decomposition must then live in the image of $\mathsf{B}(R\op{Q})(1)=\op{I}(1)\to\mathsf{B}(\op{Q})(1)$.  Thus $\eta(c)$ may be represented by a stable $\op{Q}$-labeled tree and so $\eta(c)$ is in the image of the injective map $\mathsf{B}(R\op{Q})(n)\to \mathsf{B}(\op{Q})(n)$.  This provides us with the lift $\op{C}\to \mathsf{B}(R\op{Q})$ from which the claim follows.
\end{proof}

\begin{remark}  This proposition is a second example (along with Theorem A) of an intertwining statement, i.e.\ it relates two functors by intertwining them with (bar) duality.  A systematic treatment of intertwining theorems for algebraic operads was given in \cite{Ward6ops} in the language of the six functors formalism.  Those results furnished intuition behind the results of this paper, and it would furthermore be desirable to work out the six functor formalism in the homotopy categories of generalized operads in spaces and spectra for which the results of this paper become a formal consequence.
\end{remark}

\section{Proof of the main theorem}\label{mainthmsec}
In this section we will prove our main theorem:

\begin{nonumbertheoremA} Let $\op{P}$ be a reduced $G$-operad in based topological spaces.  Then $\mathsf{B}(\op{P}\rtimes G)$ carries the structure of a cooperad for which there is a weak equivalence of cooperads:
	\begin{equation*}
	EG_+\wedge_G \mathsf{B}(\op{P}) \stackrel{\sim}\longrightarrow \mathsf{B}(\op{P}\rtimes G).
	\end{equation*}
\end{nonumbertheoremA}

Since weak and homotopy equivalences of $\mathbb{S}$-modules are given level-wise, it is enough to construct such a cooperad map which is a homotopy equivalence in each arity.
We will first construct this map and its homotopy inverse in an arbitrary arity $n$ and conclude by comparing the cooperad structures.

To construct the homotopy we will subdivide the proof in to cases based upon whether an $n$-tree has a stable vertex or not; in other words the case $n<2$ will be handled separately. 

{\bf The Case }$n=0$.  Both spaces are a point due to Assumption $\ref{pointed}$.

{\bf The Case }$n=1$.  Reduced implies $\op{P}(1)=S^0$ and so $\mathsf{B}(\op{P})(1)=\op{P}(1)\wedge (\ast/\emptyset) = S^0\wedge S^0=S^0$ regarded as a trivial $G$ space.  On the other hand, $\op{P}\rtimes G(1)=G_+$ and using the description of $\mathsf{B}$ above we know that
\begin{equation*}
\mathsf{B}(\op{P}\rtimes G)(1)=\left(\ds\bigvee_{1\text{-trees } \ut} (G_+)^{\wedge |\text{Ed(T)}|-1}\wedge w(\ut)_+\right)/ \sim
\end{equation*} 

Recall that a weighting in $w_0(\ut)$ means a branch has weight $0$, but there is only one branch in this case which must have weight $1$.  Thus the quotient by $G^{|\text{Ed}(\ut)|-1}\times w_0(\ut)$  is a quotient by the empty set and so adjoins a disjoint base point.  This explains the appearance of $w(\ut)_+$.

Now a $1$-tree is determined combinatorially by its number of edges, and in particular 
\begin{equation*}
\left(\coprod_{1\text{-trees } \ut} G^{|\text{Ed}(\ut)|-1}\times w(\ut)\right)_+
= \left(\coprod_{n\geq 0} G^{n}\times \Delta^n\right)_+
\end{equation*}

This correspondence is spelled out in Appendix $\ref{sec:diagrams}$.  Note that we allow the empty tree (after Remark $\ref{emptytreermk}$) which accounts for the $n=0$ factor in the above equation.  The fact that the identifications made in the quotient $\mathsf{B}(G_+)$ on the left hand side are precisely the same identification made when forming $BG$ as a quotient follows from Lemma $\ref{description}$.  

We thus find $\mathsf{B}(\op{P}\rtimes G)(1)\cong BG_+\cong EG_+\wedge_G S^0 \cong  EG_+\wedge_G\mathsf{B}(\op{P})(1)$, hence the claim.

\subsection{Proof of Theorem for $n\geq 2$}  We now fix $n\geq 2$.  

\subsubsection{Preliminary definitions} 

We start be giving a specialization of Corollary $\ref{blcor}$ to the case of a semi-direct product.

\begin{definition}  A $G$-marking of the set br($\Psi$) is the following data:
	\begin{itemize}
		\item  A point in $\EG$ (see Appendix $\ref{sec:diagrams}$) for each non root branch in br($\Psi$),
		\item  a point in $BG$ (which we associate to the root branch in br($\Psi$)).
	\end{itemize} 
\end{definition}

\begin{construction}\label{egmap}  Let $\op{P}$ be a reduced operad.  To each point $\Psi\in \mathsf{B}(\op{P}\rtimes G)(n)$ we associate: 
\begin{enumerate}
\item A point in $\mathsf{B}(\op{P})(n)$, and
\item a $G$-marking of br($\Psi$), 
\end{enumerate}
as follows.  First we define the ``standard representative'' of a non-base point in $\mathsf{B}(\op{P}\rtimes G)(n)$.  It is the unique weighted, labeled tree representing that point such that:
\begin{enumerate}[I.]
\item  All stable vertices are labeled by $\op{P}\subset \op{P}\times G^{\times \ast}$, as discussed above in subsection $\ref{secsdp}$.
\item  No branch has weight $0$.
\item  Every edge terminating at a stable vertex has weight $0$, and no other edge has weight $0$.
\item  No vertex is labeled by $e$, except possibly those immediately above stable vertices.
\end{enumerate}

Then for (1), we associate the base point to the base point.  Then to a non base point in $\mathsf{B}(\op{P}\rtimes G)(n)$ we associate the point in $\mathsf{B}(\op{P})(n)$ formed by first choosing the standard representative of the source and then removing all unstable vertices and their labels; see e.g.\ Figure $\ref{fig:bij}$.

For (2), observe that if $\Psi$ is the base point, there is a unique such marking (the empty marking), so we assume $\Psi$ is not the base point. Starting from the standard representative of $\Psi$ we disconnect each branch from its adjacent stable vertices, remove all $0$ weight edges, and scale the weight proportionally to $1$.  This process associates to each non root branch a point in $\EG:=B(G,G,\ast)$ and to the root branch a point in $BG$ via the diagrammatics of Appendix $\ref{sec:diagrams}$.
\end{construction}

\begin{figure}
	\centering
		\includegraphics[scale=.65]{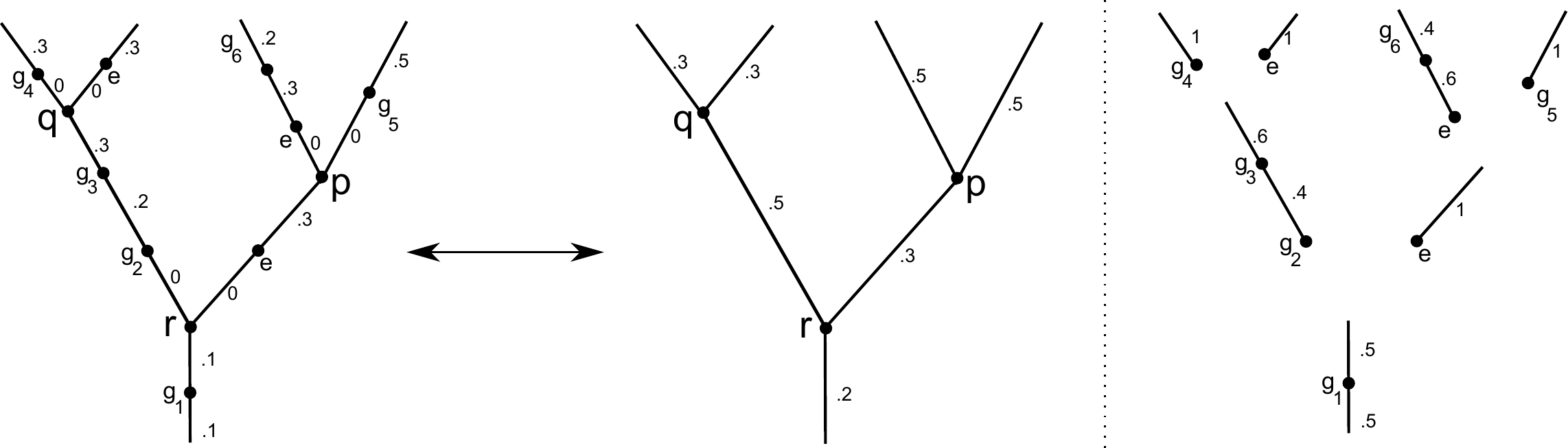}
	\caption{Depicting Construction $\ref{egmap}$.  Left is a point in $\mathsf{B}(\op{P}\rtimes G)(n)$ depicted via the standard representative.  Here $p,q,r\in \op{P}(2)$ and $g_i\in G$.  The construction associates to this the data on the right; a point in $\mathsf{B}(\op{P})(n)$ (middle right) and a $G$-marking (far right), depicted after Appendix $\ref{sec:diagrams}$.}
	 \label{fig:bij}
\end{figure}

\begin{lemma} \label{marklem} Let $\op{P}$ be a reduced $G$-operad.  Construction $\ref{egmap}$ establishes a bijection (of sets) between $\mathsf{B}(\op{P}\rtimes G)(n)$ and the set of pairs $(\psi, G\text{-markings}(\psi))$ for $\psi\in \mathsf{B}(\op{P})(n)$.
\end{lemma}
\begin{proof}  
We use the standard representative.  The construction gives a map in one direction.  For the inverse, suppose we have a $G$-marking on the set br($\psi$), where $\psi\in \mathsf{B}(\op{P})$ (not the base point).  Choose the minimal representative of $\psi$; a stable tree with no zero weight branches, as well as the minimal representatives of the $G$-marking; $1$-trees (Appendix $\ref{sec:diagrams}$) associated to each branch with no zero weight edges.  We then simply paste the $1$-trees over the branch, scaling to the (necessarily non zero) weight and connecting below with a $0$ weight for the non root branches. See Figure $\ref{fig:bij}$.

The fact that this correspondence is bijective follows from the fact that the compositions are the identity in both ways.  Note also it takes base point to base point.
\end{proof}

\begin{remark}\label{notationrmk}  This lemma will allow us to depict points in $\mathsf{B}(\op{P}\rtimes G)(n)$ as lists $\Psi=(\st,p_\ast;\beta)$.  Here $\st$ is a stable weighted tree; the weights are part of the data.  We emphasize that $p_\ast= \wedge_{V(\st)} p_i$ takes values $p_i\in \op{P}(v_i)$.  In particular, the data $(\st, p_\ast)$ determine the point $\psi$ in $\mathsf{B}(\op{P})$ specified in (2) above.  And finally $\beta$ is a $G$-marking of $(\st, p_\ast)$.  This depiction is unique up to identifications in $\mathsf{B}(\op{P})$.  In particular if we insist that $\st$ be minimal (no $0$ weight branches) then each (non-base) point is uniquely specified by such a list.  Therefore, from now on when we use the notation $(\st,p_\ast;\beta)$, we will choose $\st$ to be in minimal form.

We emphasize that the bijective correspondence established by Lemma $\ref{marklem}$ is one of sets and is not in any sense continuous.
\end{remark}

\begin{definition}\label{augGdef} An augmented $G$-marking of the set br($\Psi$) is the following data: 
\begin{itemize}
		\item  a point in $B(G,G,G)$ for each internal branch in br($\Psi$),
		\item  a point in $\widetilde{EG}$ for each leaf branch in br($\Psi$), and
		\item  a point in $EG$ (associated to the root branch in br($\Psi$)).
	\end{itemize} 
Given a $G$-marking, call it $\beta$, of the set br($\Psi$), we define an augmented $G$-marking, call it $\beta_e(-)$, by placing $e$ in the right module position of $\widetilde{EG}$ (resp.\ $BG$) presented in the minimal form.  See e.g.\ the right hand side of Figure $\ref{fig:pi}$.
\end{definition}

We remark that given a point in $\mathsf{B}(\op{P})(n)$ and an augmented $G$-marking, we can construct a point in $\mathsf{B}(\op{P}\rtimes G)(n)$ as above, by connecting the right module label $e$ to a stable vertex above by a $0$ weight.  In this case, the construction is not bijective and the analog of Lemma $\ref{marklem}$ will not hold.  The failure to be bijective is encoded by the semi-direct product identifications (Equation $\ref{sdp}$).

\begin{definition} \label{gvdef} Given $\Psi=(\st,p_\ast,\beta)\in \mathsf{B}(\op{P}\rtimes G)$ we define a map $V(\st)\to G$ by $g_v:= \prod \mu(\beta_e(E))$ where the product is taken over all branches $E$ connecting the vertex $v$ to the root vertex of $\st$ (via the unique directed path, so as to form an ordered product) and $\mu\colon B(G,G,G)\to G$ is the multiplication map.
\end{definition}

\subsubsection{The homotopy equivalence.}  

We are now prepared to define the continuous map that will prove to be our homotopy equivalence between these two spaces, called $\sigma$:
\begin{equation*}
\sigma\colon  EG_+\wedge_G  \mathsf{B}(\op{P})(n) \to\mathsf{B}(\op{P}\rtimes G)(n).
\end{equation*}
To begin we will need the following definition.

\begin{definition} \label{gammadef}  Define $\gamma$ to be the map from $EG$ to $B(G,G,G)$ which takes a point in $EG$ and places the inverse of the product of the vertex labels in the left module position.  For example the point in $EG$ pictured below on the left (after Appendix $\ref{sec:diagrams}$) is sent to the point in $B(G,G,G)$ pictured below on the right:
	\begin{center}
		\includegraphics[scale=1.5]{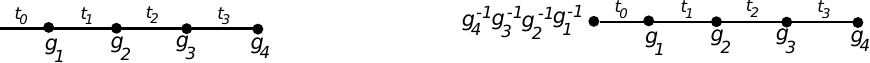}
	\end{center}
\end{definition}

We now define $\sigma([\zeta\wedge (\st,p_\ast)]):= (\st,p_\ast,\beta)$ where $\beta$ is defined by:
\begin{align*}
\beta(E)=\begin{cases} 
\zeta & \text{if } E \text{ is the root branch,} \\  
\text{pr}_r(\gamma(\zeta)) & \text{if } E \text{ is a leaf branch,} \\
\gamma(\zeta) & \text{if } E \text{ is an internal branch.}
\end{cases} 
\end{align*}

Note that here $\beta$ is an augmented $G$-marking (Definition $\ref{augGdef}$) and pr$_r\colon B(G,G,G)\to B(G,G,\ast)$ is projection on the right induced by $G\to \ast$.

\begin{figure}
	\centering
	\includegraphics[scale=.85]{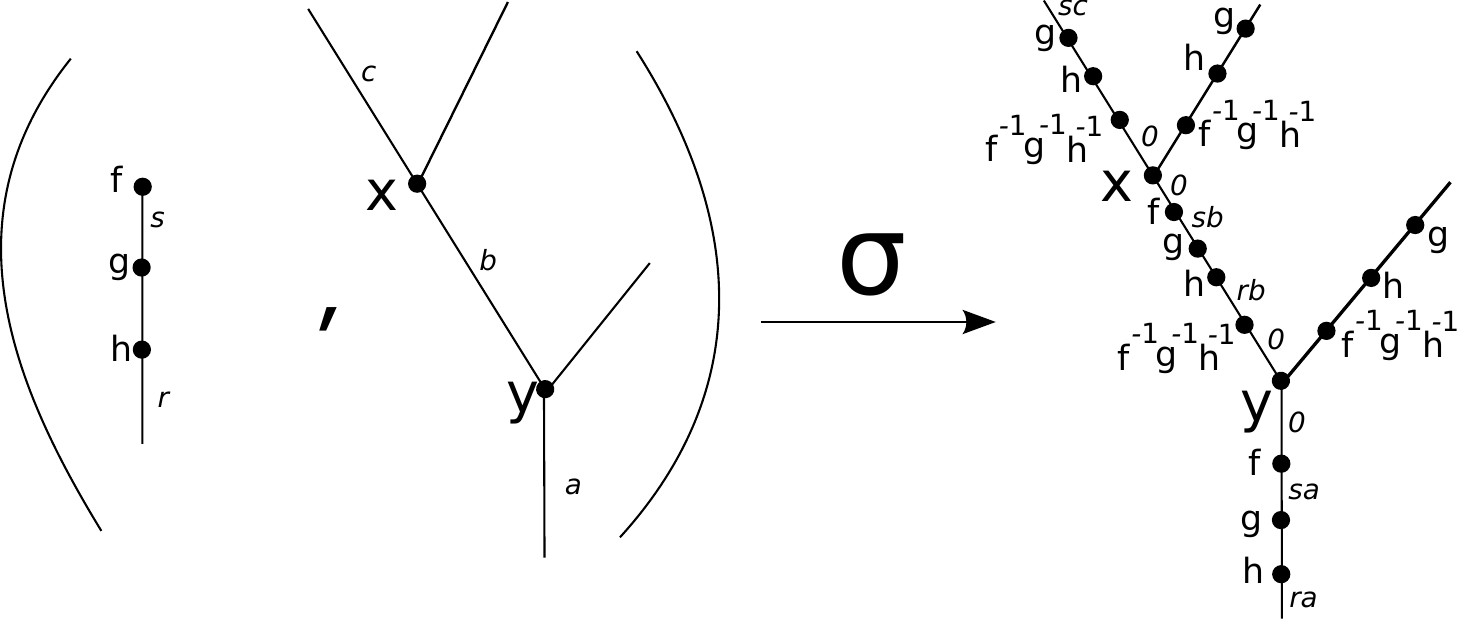}
	\caption{The map $\sigma$.  Here $a,b,c,s,r$ are weights and $f,g,h \in G$.  Note that not all weights are depicted and the leaf labeling is suppressed.}
	\label{fig:sigma}
\end{figure}

\begin{lemma}  As defined above, $\sigma$ is a well defined continuous map of based spaces.
\end{lemma}

\begin{proof}  Since we assume the input is in minimal form, to check $\sigma$ is well defined as a map of sets we need only check that the definition is independent of a choice of representative of $G$ coinvariants.  This follows immediately from the SDP identifications in the target.	
	
	So it remains to argue that $\sigma$ is continuous.  We argue (as below in the proof of Lemma $\ref{paths}$) that it is enough to consider paths in the weights, and the only possible ambiguity arises if branch weights go to zero or if an edge (or edges) in the factor of $EG$ go to zero in the source.  We can check these cases by hand.
	
	If a branch weight goes to $0$ either it is a leaf or root branch (in which we converge to the base point in both the source and target) or it is an internal branch.  In this latter case, we may compose along this branch in the target by contracting $0$ weight edges and multiplying the adjacent unstable vertex labels.  Since the product of the target's unstable vertex labels on this branch is the identity, this is the same as taking the limit in the source before applying $\sigma$.

	Now consider what happens if an edge weight in $\zeta$ converges to zero.  For an edge which is not the root edge we observe that contracting the corresponding $0$ weight edges before or after applying $\sigma$ yields the same result.  For example, letting $s\to 0$ in Figure $\ref{fig:sigma}$.  If it is the root edge, we see this has the effect in the limit of the target of multiplying an element by its inverse which corresponds to the element having disappeared in the target of the limit.  For example, letting $r\to 0$ in Figure $\ref{fig:sigma}$. \end{proof}

\subsubsection{The homotopy inverse.} 
In this subsection we define a retraction of $\sigma$, which we call $\pi$:
\begin{definition}\label{pidef}  Define a map 	
	\begin{equation*}
	\pi\colon \mathsf{B}(\op{P}\rtimes G)(n) \to EG_+\wedge_G \mathsf{B}(\op{P})(n)
	\end{equation*}
by $\pi(\st,p_\ast,\beta):= [\beta_e(R)\wedge (\st,g_{v_\ast}p_\ast)]$, where $R$ is the root branch of $\st$.  See Figure $\ref{fig:pi}$.
\end{definition}

\begin{figure}[h]		
	\centering 	
	\includegraphics[scale=.75]{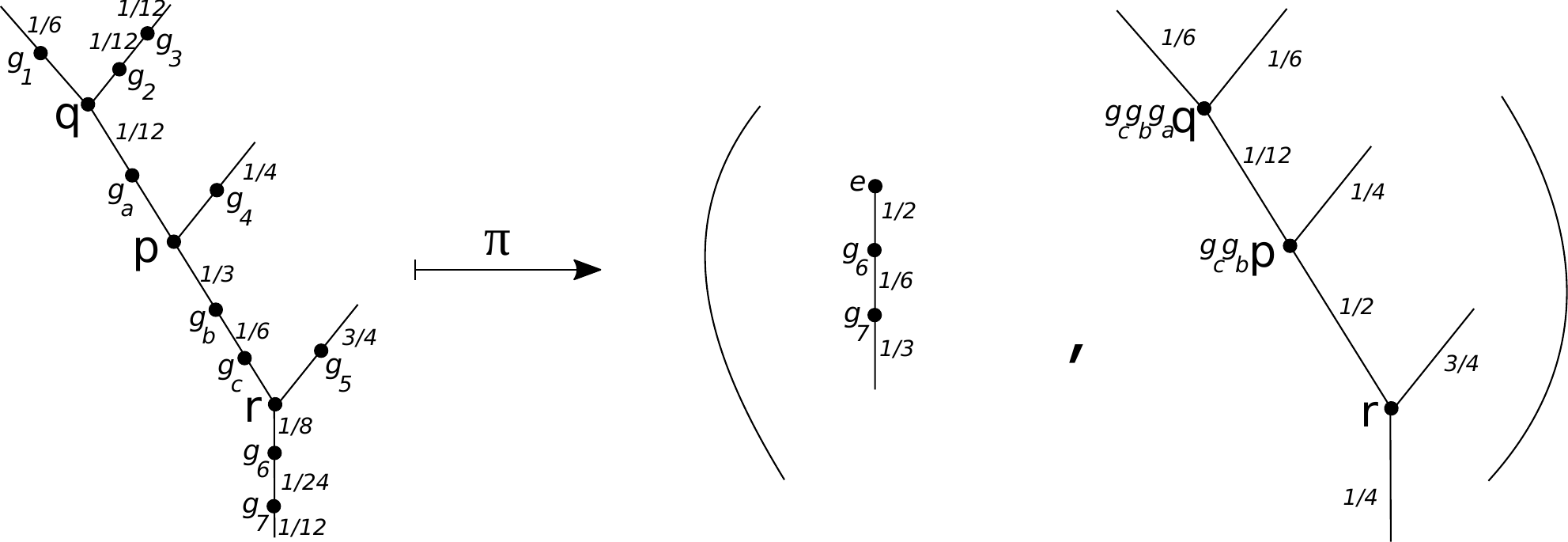}
	\caption{Graphical description of the map $\pi$, with $p,q,r\in\op{P}(2)$ and $g_i\in G$.  Note that $0$ weights and leaf labels are not depicted. }
	\label{fig:pi}
\end{figure}

\begin{lemma}\label{paths} $\pi$ is a continuous, $S_n$ equivariant retraction of $\sigma$.
\end{lemma}
\begin{proof} 
Since the notation $(\st, p_\ast, \beta)$ depicts (non base) points in $\mathsf{B}(\op{P}\rtimes G)(n)$ uniquely (after  Remark $\ref{notationrmk}$), $\pi$ is obviously a well defined map of sets.  It is also immediate to see that $\pi$ is based and $S_n$ equivariant.  To say $\pi$ is a retraction of $\sigma$ means that $\pi\circ\sigma=id$, and this in turn is just the statement that the product of $G$ labels from each vertex to the least stable vertex is the identity in the target.


So it remains to show that $\pi$ is continuous.  Since the source of $\pi$ is a quotient of a coproduct of a quotient,  it is enough to show that for an arbitrary $n$-tree $\ut$, the map
\begin{equation}\label{picont}
\op{P}\rtimes G(\ut)\times w(\ut)\to EG_+\wedge_G \mathsf{B}(\op{P})
\end{equation}
is continuous.  For this, we fix a convergent sequence $x_j\to x$ in the source of line $\ref{picont}$.  Since $\ut$ is fixed, so is the set of edges and we may consider the convergent sequence of edge weights associated to any edge $E\in \text{Ed}(\ut)$.  If no such sequence converges to zero, then $\pi$ obviously achieves its limit on this sequence.  To see this, using the fact that there are finitely many edges, we may pick an $\epsilon >0$ and an integer $N$ such that if $j>N$ then all edge weights are at least $\epsilon$.  Thus the image of $\pi$ can be represented by the same diagram with continuously changing labels after stage $N$.

It now remains to consider the case that one or more edge weights converge to zero.  Let us use the short hand vernacular ``an edge goes to zero'' to describe this phenomenon.  We first observe that if an internal branch goes to zero we may compare the limit of the image of $\pi$ with the image of the limit using the equivariance of the operadic compositions ($\op{P}$ is a $G$-operad by assumption).  If an external branch goes to zero, the sequence converges to the base point in both the source and the target.

Next we consider an edge going to zero within a branch which does not go to zero.  The only case which does not follow from an immediate comparison of the unstable weighted labeled trees representing the limit of the image and the image of the limit, is when an edge (or edges) immediately below a stable vertex goes to zero.  If said vertex is not the least stable vertex, we may use the semi-direct product identifications to see that the vertex labels of $\pi$ of the limit are still given by the limit of the product of all unstable vertex labels connecting $v$ to the least stable vertex.

At the same time, if the edge (or several consecutive edges) immediately below the least stable vertex converges to $0$, we use both the SDP identifications and coinvariants to see that $\pi$ preserves the limit.  That is, if we let $h\in G$ be the product of those consecutive unstable vertex labels below the least stable vertex which converge to $0$, then to compare $\pi(x)$ with the limit of $\pi(x_j)$ we use coinvariants to move the additional factor of $h$ appearing on each vertex in $\pi(x)$ (owing to the semi-direct product identifications in the pre-image) to the $EG$ factor to find the limit of $\pi(x_j)$. \end{proof}

\subsubsection{Construction of the homotopy}

\begin{theorem}  The maps $\sigma$ and $\pi$ are homotopy equivalences.
\end{theorem}
\begin{proof}  Since $\pi\circ\sigma=id$, this result will follow from the construction of:
\begin{align*}
H\colon I\times \mathsf{B}(\op{P}\rtimes G)(n)\to\mathsf{B}(\op{P}\rtimes G)(n) \\ 
\text{such that  } H(0,-)= id \text{ and } H(1,-)= \sigma\circ\pi.
\end{align*}
Before giving the formal definition of $H$ we provide an informal description.

{\bf Informal description of $H$:}  Let $s$ denote the $I$ parameter and consider the action of $H(s,-)$ on the space of $\op{P}\rtimes G$-labeled weighted trees.  Since this is an informal description we'll call $s$ the water level and we imagine the water level rising.  The rising water level has no effect until it reaches the least stable vertex.  As the water passes this vertex it begins to contract the copy of $\widetilde{EG}$ corresponding to every branch that is partially submerged at the speed dictated by the branch weights.  The target of this contraction is dictated by the root labeling; specifically it sends $g_{v^E}\beta_e(E)$ to $\gamma(\beta_e(R))g_{v_E}\in B(G,G,G)$ in $|E|$ units of time (here $R$ is the root branch of the given point and $\beta_e$ is the augmented $G$-marking as in Definition $\ref{augGdef}$).  The route for this contraction is a line segment formed by grafting and scaling the graphical representation of these points.  We continue this process until the tree is completely submerged.

{\bf Formal definition of $H$:}  
Fix $\Psi \in \mathsf{B}(\op{P}\rtimes G)(n) $ and write $\Psi=(\st,p_\ast,\beta)$ (after Remark $\ref{notationrmk}$).  We will define $H(s,(\st,p_\ast,\beta)):=(\st,p_\ast^s,\beta^s)$.  Observe that $H$ preserves the underlying weighted stable tree $\st$.  We have now to define $p_\ast^s$ and then $\beta^s$.  We disclose that this is an abuse of notation in that $p_\ast^s$ will depend not just on $p_\ast$ and $s$ but on $\beta$ as well.

For a vertex $v\in V(\st)$, recall (Definition $\ref{gvdef}$) that $g_v$ denotes the product of the unstable vertex labels between $v$ and the root vertex of $\st$.  We then define $p_\ast^s=\wedge_{v\in V(\st)} p_v^s$ where
\begin{equation}\label{HVert}
p_v^s:=\begin{cases} g_v p_v & \text{if $v$ is no longer active at $s$} \ \ (\text{Definition } \ref{wtterms}), 
\\ p_v &  \text{ else.} \end{cases}
\end{equation}

To define $\beta^s$ we need an auxiliary definition.  For a non-root branch $E$ of $\st$ we define $\ell_E\colon [|v^E|,|v_E|]\to B(G,G,G)$ to be the parametrized line segment connecting $g_{v^E}\beta_e(E)$ and $\gamma(\beta_e(R_\Psi))g_{v_E}$ by grafting $g_{v^E}\beta_e(E)$ above $\gamma(\beta_e(R_\Psi))$, and linearly scaling the weights in $|E|$ units of time.  (See the green branch of Figure $\ref{fig:ho}$.)  We then define:
\begin{equation*}
  \beta^s(E):=\begin{cases}
  
\beta(E) & \text{if $E$ is not yet active at $s$ or if $E$ is the root edge,} \\  
 \ell_E(s) & \text{if $E$ is active at $s$ and not the root edge,} \\ 
\gamma(\beta_e(R))  & \text{if $E$ is no longer active at $s$ and not the root edge.} \end{cases}
\end{equation*}

Here we are using the notation and terminology from Definitions $\ref{wtterms}$ and $\ref{gammadef}$.  Observe in particular that $g_{v^E}\beta_e(E)$ is the branch label of $E$ at the instant $E$ becomes active.  In the case that $E$ is a leaf we may equivalently view this construction as specifying a parametrized line segment in $\widetilde{EG}$.

\begin{figure}[h]		
	\centering 	
	\includegraphics[scale=.4]{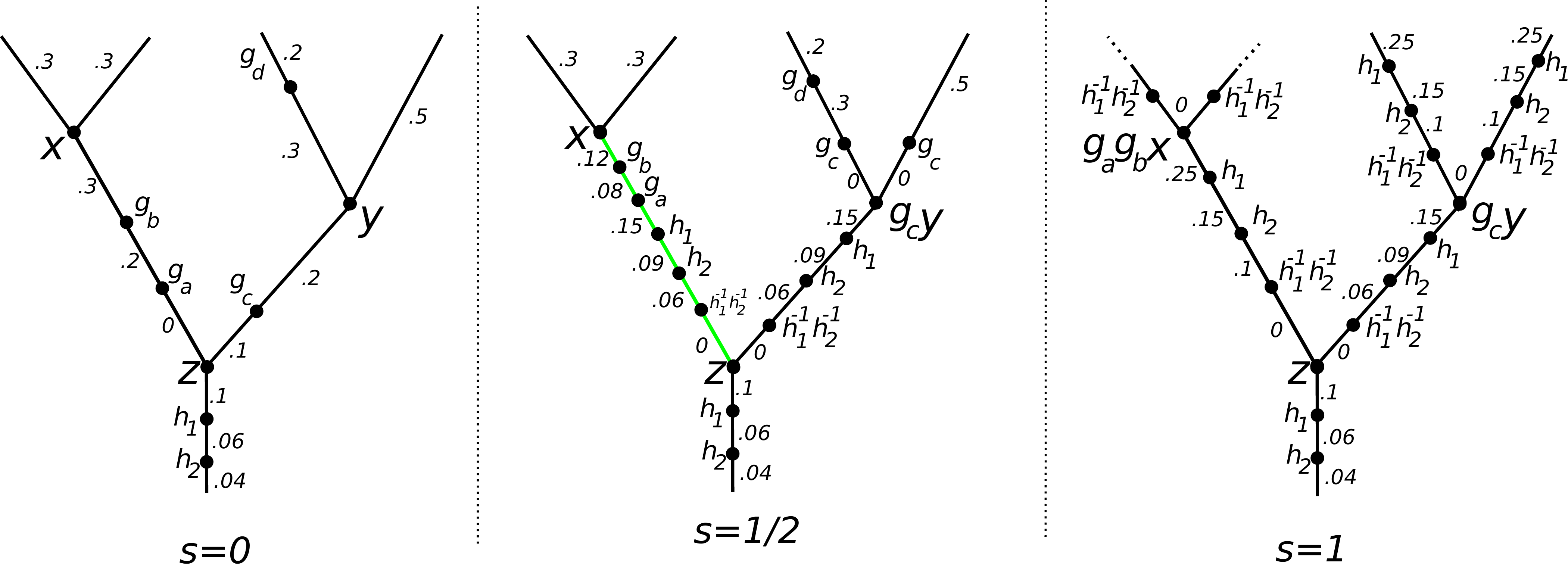}
	\caption{Graphical description of the homotopy $H$. Here $g_i,h_j\in G$ and $x,y,z\in \op{P}(2)$.  At time $s=1/2$ the green branch is being actively contracted. }
	\label{fig:ho}
\end{figure}

The fact that $H$ is well defined as a map of sets follows from the fact that it was defined using a choice of standard representative.  To argue that $H(1,-)=\sigma\circ\pi$ observe that, since $s=1$, everything is no longer active.  This means we have replaced the marking of all non-root edges with $\gamma(\beta_e(R))$, while also multiplying the stable vertex labels by $g_v$ for each vertex $v$.  This is exactly the description of $\sigma\circ\pi$.

So it now remains to show that $H$ is continuous.  Since the source of $H$ is a quotient space, it is enough to show that its composite with the quotient map is continuous.  The source of this map is itself the quotient of a coproduct over $n$-trees.  So we fix a tree $\ut$ and let $\tilde{H}$ denote the induced map: 
$I\times(\op{P}\rtimes G)(\ut)\wedge \bar{w}(\ut)\stackrel{\tilde{H}}\longrightarrow \mathsf{B}(\op{P}\rtimes G)(n)
$, and it is enough to show $\tilde{H}$ is continuous.

Pick a sequence of points $(s_j,\Psi_j)$ converging to $(s,\Psi)$.  Then $s_j \to s$, $|E|_j\to |E|$, and $|v|_j\to |v|$, where $|E|_j$ is the weight of edge $E$ at stage $j$, $|v|_j$ is the altitude of a vertex $v$ at stage $j$, and no subscript denotes the weight (resp.\ altitude) in the limit.  In particular, for each vertex the sequence $s_j-|v|_j$ converges (to $s-|v|$).

We now argue $lim_j\tilde{H}(s_j,\Psi_j)=\tilde{H}(s,\Psi)$ by considering two cases.  First suppose that for every stable non-root vertex $v_i$ we have $s-|v_i|\neq 0$.  Then, since there are finitely many stable vertices, and since $s_j-|v_i|_j$ converges to some (necessarily non-zero) number, we may choose $\epsilon>0$ and an integer $N$ such that $|s_j-|v_i|_j|>\epsilon$ for all $j>N$ and for all stable vertices $v_i$.  This means for all $j>N$, $\tilde{H}(s_j,\Psi_j)$ is represented by the same unweighted, unstable tree (not necessarily equal to $\ut$).  

Let us consider two subcases. If no edges converge to zero in the source (i.e.\ if $|E|\neq 0$, for each $E\in \text{Ed}(\ut)$), then $lim_j\tilde{H}(s_j,\Psi_j)$ is in minimal form and so we may compare this point with $\tilde{H}(s,\Psi)$ simply by comparing the vertex labels and edge weights.  Since the group action is continuous, the sequence of stable vertex labels in the target is convergent.  Since the weights of edges in the target are all determined by scaling the edges in the source, each sequence of edge weights in the target is continuous.  Whence the first subcase.

Next, suppose there is a non-empty set of edges $\{E_z\}$ whose edge weights $|E_z|_j$ converge to zero.  Then $\tilde{H}(s,\Psi)$ is defined via the tree having those edges contracted.  It is sufficient to show that using identifications in the bar construction to contract the zero edges in the weighted labeled tree defining $lim_j\tilde{H}(s_j,\Psi_j)$ yields the weighted labeled tree defining $\tilde{H}(s,\Psi)$, and we can check this combinatorially.

If an entire branch converges to zero in the source, then clearly the same is also true in the target so we may assume this is not the case.  If an edge immediately below a stable vertex which is not yet active converges to zero, this follows from continuity of the original sequence (since not yet active means not yet modified).  If an edge immediately below the least stable vertex converges to zero, this follows from the SDP identifications.  If the root edge converges to zero, this follows from the definition of $\gamma$.  If an edge not adjacent to a stable vertex and not the root converges to zero, this follows from continuity of the group multiplication.  Whence the second subcase.


The other case is that $s-|v|=0$ for a (or several) stable non-root vertex $v$.  To show $\tilde{H}$ achieves its limit we can consider possible subsequences having $s_j-|v|_j>0$ and $s_j-|v|_j<0$ respectively.  The fact that $\tilde{H}$ would achieve its limit on the former subsequence follows from the same argument as above; it can be represented (for sufficiently high $j$) by a single underlying unstable tree with the labels changing in a continuous fashion.

Thus it suffices to consider a sequence in which $s_j-|v|_j$ converges to $0$ from below, and so we now consider $(s_j,\Psi_j)$ to be such a sequence.  In this case, the point $lim_j \tilde{H}(s_j,\Psi_j)$ is represented by a weighted $\op{P}\rtimes G$-labeled tree whose vertex $v$ is labeled by the limit of the vertex labels, but immediately below $v$ are unstable vertices connected to $v$ by $0$ weights whose  product is $g_v$.  On the other hand we see $\tilde{H}(s,\Psi)$ has a factor of $g_v$ on and immediately above this vertex connected by $0$ weights.  But the semi direct product identifies these two expressions in the bar construction $\mathsf{B}(\op{P}\rtimes G)$, and so $\tilde{H}$ achieves its limit on such a sequence.\end{proof}

Having constructed our level-wise equivalence of $\mathbb{S}$-modules, we conclude the proof of the main theorm by observing:

\begin{corollary}  With respect to the above cooperad structure, the homotopy equivalence $\sigma$ is a morphism of cooperads $EG_+\wedge_G\mathsf{B}(\op{P})\to\mathsf{B}(\op{P}\rtimes G)$.
\end{corollary}
\begin{proof}  This is a straight forward diagram chase which we briefly describe.  Consider $[\zeta\wedge (\st, p_\ast)]\in EG_+\wedge_G\mathsf{B}(\op{P})(n)$.  The non trivial degraftings correspond to edges of $\st$ so we fix such an edge $E\in \text{Ed}(\st)$.  If we first degraft at $E$ and then apply $\sigma$, the cooperad structure degrafts the tree and uses the diagonal on the $EG$ factor, labeling both root branches by $\zeta$.  Applying $\sigma$ means non root branches will be labeled with $\gamma(\zeta)$ on both the left and the right factors.  On the other hand if we first apply $\sigma$ before degrafting, the non-root branches become labeled with $\gamma(\zeta)$ before degrafting.  The nontrivial degraftings still correspond to the edges of $\st$, and degrafting at $E$ will preserve these branch labels, however on the new root branch of the right hand factor, the root edge will have $0$ weight and so can be removed.  This corresponds to removing the left module of $\gamma(\zeta)$ which gives back $\zeta$ as desired. \end{proof}


\section{Discussion and future directions}\label{discsec}

To conclude we will consider several implications and future directions of our main theorem.  Our principal aim is to show how the results of this paper may be combined with a conjectural equivariant self-duality of the little disks in spectra to better understand Koszul duality between the moduli spaces of punctured spheres and their Deligne-Mumford compactifications.  For this, we first recall some background.
 
The space $\op{M}_n$ is configurations of $n$ points on a sphere modulo the action of M{\"o}bius transformations.  The space $\overline{\op{M}}_{n}$ is a compactification of $\op{M}_n$ allowing for nodal configurations of points.  The homology of $\overline{\op{M}}_{\ast+1}$ is a Koszul operad whose Koszul dual is (a degree shift of) the homology of the open moduli space $\Sigma H_\bullet(\op{M}_{\ast+1})$.  The operad structure of $H_\bullet(\overline{\op{M}}_{\ast+1})$ arises from the topological operad $\overline{\op{M}}_{\ast+1}$ by freely identifying marked points to create a new node.  On the other hand the operad $\Sigma H_\bullet(\op{M}_{\ast+1})$, called the gravity operad by Getzler \cite{Geteq}, does not arise directly as the homology of a topological operad.

The topological operad $\overline{\op{M}}_{\ast+1}$ is related to the framed little disks, $fD_2$, by the functor $L\colon \{\text{operads}\}\to \{\text{reduced operads}\}$ which was introduced in subsection $\ref{comparisonsec}$.

\begin{theorem}\cite{DC}\label{DC}  Let $\widetilde{L}$ denote the total left derived functor of $L$.  Then $\widetilde{L}(fD_2)\sim \overline{\op{M}}_{\ast+1}$. \end{theorem}

A model for $\widetilde{L}(fD_2)$ is given by $L(\widetilde{fD}_2)$ for any cofibrant replacement $\widetilde{fD_2}\stackrel{\sim}\to fD_2$.  To connect this to our main result we first prove the following lemma.  It uses the notation $L$ and $\widetilde{R}$ from Section $\ref{comparisonsec}$.
\begin{lemma}\label{lemm}  The operads $\Omega (\mathsf{B}(D_2)_{S^1})$ and $L (\Omega  \mathsf{B}(fD_2))$ are weakly equivalent.
\end{lemma}
\begin{proof}
Using our main theorem, there is a homotopy equivalence of cooperads $\mathsf{B}(D_2)_{hS^1}\stackrel{\sim}\to \mathsf{B}(fD_2)$.  Applying the cooperadic reduction $\widetilde{R}$ we have homotopy equivalences:
\begin{equation*}
\mathsf{B}(D_2)_{S^1}\stackrel{\sim}\leftarrow \widetilde{R}\mathsf{B}(D_2)_{hS^1}\stackrel{\sim}\to \widetilde{R}\mathsf{B}(fD_2)\end{equation*}
  Here we have used the fact that the $S^1$ action is free in arity $\geq 2$, and we dropped the $\widetilde{R}$ notation on the left hand side since $\mathsf{B}(D_2)_{S^1}(1)$ is already trivial.

We then observe that the bar and cobar constructions take levelwise homotopy equivalences to weak equivalences.  This follows as in \cite{AC} Proposition 8.5 (see also Remark 8.6), but in our case we do not need to take termwise cofibrant replacements, since we started with an actual homotopy equivalence.  
It follows that $\Omega (\mathsf{B}(D_2)_{S^1})$ and $ \Omega(\widetilde{R}  \mathsf{B}(fD_2))$ are weakly equivalent.  
 
Finally we use Proposition $\ref{redint}$ to exchange $\Omega \widetilde{R}$ with $L \Omega$ from which the claim follows.
\end{proof} 

Lemma $\ref{lemm}$ motivates us to consider the implications of our results in a category in which the bar-cobar construction can be used to produce a resolution of the input.  From Ching's result in \cite{Ching2}, one knows this is the case in certain categories of spectra, and so we conclude by examining our results in that context.

\subsection{Passing to Spectra}

To this point we have considered the bar construction for topological operads following Ching \cite{Ching}.  But loc.cit.\ defines this cooperad structure for more general base categories, namely for operads valued in any symmetric monoidal category suitably enriched, tensored and cotensored over Top$_\ast$.  This encompasses not only Top$_\ast$ itself but also suitable categories of spectra.  We begin by fixing such a category of spectra and recalling several relevant properties.

Following \cite{Ching2} we let Spec be the category of S-modules of \cite{EKMM}.  It is a closed symmetric monoidal category which receives a strong symmetric monoidal functor $\Sigma^\infty\colon $Top$_\ast\to$ Spec which is a left adjoint.  We denote the unit by $S$, the monoidal product by $\wedge$, and the internal hom functor by $F_S(-,-)$.  We will use the fact that $F_S(\Sigma^\infty X,M)$ can be written in terms of mapping spaces (which we denote $F(X,M)$ after loc.cit.) and refer to \cite{EKMM} Proposition II.1.4 for the precise statement.

If $\op{P}$ is an operad in spaces we write $\Sigma^\infty\op{P}$ for the operad in spectra obtained by levelwise application of $\Sigma^\infty$.  We also remark that the triple of adjoint functors $(L,\iota, R)$ relating operads and reduced operads in spaces (see Section $\ref{comparisonsec}$) may also be constructed between operads and reduced operads in spectra.  Since $\Sigma^\infty$ is a strong monoidal left adjoint, it preserves both colimits and $\wedge$.  It follows that $L$ and $\mathsf{B}$ commute with $\Sigma^\infty$.  We furthermore observe $\iota$ and $R$ commute with $\Sigma^\infty$ by inspection.

\subsubsection{Derived Koszul dual $K$, after \cite{Ching2}}  Working in spectra gives us access to the results of \cite{Ching2} and a homotopy involutivity result about the bar-cobar construction.  Moreover this homotopy involutivity result may be phrased purely in the language of operads using Spanier-Whitehead duality.

To recall this result we first define $\mathbb{D}(-):=F_S(-,S)$.  This is a contravariant functor equipped with a natural map $\mathbb{D}(X)\wedge \mathbb{D}(Y)\to \mathbb{D}(X\wedge Y)$, and so takes cooperads to operads (\cite{Ching2} Definition 4.3). We then define $K(-):=\mathbb{D}(\mathsf{B}(-))$.   Our main theorem may be rephrased in this language:

\begin{corollary}\label{cor4}  Let $\op{P}$ be a reduced operad in based $G$-spaces.  There is a weak equivalence of operads in spectra:
	\begin{equation*}
	K(\Sigma^\infty\op{P})^{hG}\sim K(\Sigma^\infty(\op{P}\rtimes G))
	\end{equation*}
\end{corollary}
\begin{proof}  Proposition II.1.4 of \cite{EKMM} allows us to write $F_S(\Sigma^\infty\mathsf{B}(\op{P}\rtimes G),S)$ in terms of the spectrum of mapping spaces $F(\mathsf{B}(\op{P}\rtimes G),S)$. Then we may use the fact that $F(-,S)$ preserves levelwise homotopy equivalences, along with our main theorem, to apply Theorem I.8.5 of \cite{EKMM} and conclude $F_S(\Sigma^\infty\mathsf{B}(\op{P}\rtimes G),S)\sim F_S(\Sigma^\infty\mathsf{B}(\op{P})_{hG},S)$, from which the claim follows. \end{proof}

	

One advantage of viewing our result in this language is we have access to Theorem 4.11 of \cite{Ching2} which shows that if $\op{P}$ is a reduced operad in spectra which is levelwise cofibrant and which satisfies suitable finiteness hypotheses then
\begin{equation*}
K\widetilde{K}(\op{P})\sim\op{P}
\end{equation*}
where $(\widetilde{-})$ denotes levelwise cofibrant replacement.

We now recall two conjectures from \cite{Ching2}.  We adapt the notation $\Sigma^\infty_+$ to denote $\Sigma^\infty ( (-)_+)$ so as to emphasize that our operads of interest here are {\it a priori} non-based.

{\bf Conjecture 5.4 of} \cite{Ching2}:  The operad $K(\Sigma^\infty_+D_n)$ is equivalent to a suitable operadic suspension of $\Sigma^\infty_+D_n$ in the category of spectra. 

{\bf Conjecture 5.5 of} \cite{Ching2}:  The operad $K(\Sigma^\infty_+ \overline{\op{M}}_{\ast+1})$ is equivalent to an operadic suspension of $\Sigma^\infty_+ R(D_2)^{S^1}$.

 Technically, the statement in \cite{Ching2} invokes the transfer operad of the little disks \cite{Westerland}, but this is equivalent to $\Sigma^\infty_+D_2^{S^1}$ as a reduced operad (Corollary 2.8 of loc.cit).

\subsubsection{Equivariant conjecture}

We would now like to tie the above conjectures together with our main theorem.  The first ingredient is a strengthening of Ching's Conjecture 5.4 in the case $n=2$:

\begin{conjecture}\label{conj} The operad $K(\Sigma^\infty_+D_2)$ is equivalent to a suitable operadic suspension of $\Sigma^\infty_+D_2$ in the category of $S^1$-spectra. 
\end{conjecture}

For evidence we recall that Conjecture 5.4 was based on an analogy with the algebraic setting, and we may update this analogy to include the recent results Theorem 1.1 of \cite{KWill} and Theorem 2.8 of \cite{WRC}, which show formality can be made equivariant in the algebraic setting when $n=2$.

The importance of Conjecture $\ref{conj}$ is that it implies not only Conjecture 5.4, but the results of this paper suggest that it implies Conjecture 5.5 as well. 
To see this, we first apply Corollary $\ref{cor4}$ to conclude:
\begin{equation*}
K(\Sigma^\infty_+\op{D}_2)^{hS^1}\sim K(\Sigma^\infty_+fD_2).
\end{equation*}
Combining this with Conjecture $\ref{conj}$ would imply that the derived Koszul dual of the stabilization of the framed little disks is (up to suspension) Westerland's spectral model for the (non-reduced) gravity operad.  

We may then pass to the reduction by taking $KR(\widetilde{-})$ of both sides and and interchanging $KR$ with  $LK$.  It remains to conclude $LK\widetilde{K}(\Sigma^\infty_+\widetilde{fD}_2)$ is a derived model for $L(\Sigma^\infty_+\widetilde{fD}_2)$, and this would follow from a suitable generalization of the homotopy theory of bar-cobar duality from the reduced to the non-reduced case, in particular using a non-reduced version of Theorem 4.11 of \cite{Ching2}.

Let us remark that, following \cite{KWill} in the algebraic case, it would be reasonable to state Conjecture $\ref{conj}$ for all $D_{2n}$ with their $SO(2n)$-action.  Moreover, following Section 6.2 \cite{Westerland}, we could consider $D_4$ with its $SU(2)$-action.  Westerland defines the four dimensional gravity operad $Grav^4$ as the homotopy transfer of the $SU(2)$ action on $D_4$.  In parallel with the above discussion, our result may be used to compare the derived Koszul dual of $Grav^4$ with the $SU(2)$-framed little four disks operad.


So in summary, the results of this paper may be combined with the following interesting future directions to augment our understanding of Koszul duality and moduli spaces:

\begin{enumerate}
	\item Generalize the homotopy theory of bar-cobar constructions in spaces and spectra \cite{Ching,AC,Ching2} from reduced to non-reduced operads and cooperads.
	\item Give a direct topological proof that the cohomology of the cooperad $\mathsf{B}(D_2)/S^1$ is the gravity operad.
	\item Prove that the Spanier-Whitehead dual of the stabilization of $\mathsf{B}(D_2)$ is (up to suspension) the stabilization of $D_2$ {\it in the category of $S^1$-operads}.
\end{enumerate}


\appendix

\section{Diagrammatic $EG$ and $BG$}\label{sec:diagrams}  The combinatorial description of the operadic bar construction generalizes a combinatorial description of the bar construction of a topological monoid in terms of $1$ trees.  We depict such trees horizontally with the convention that the leaf is on the right and the root is on the left.  Since this description is used frequently in the body of the text, we spell out these elementary considerations here.

Let $G$ be a topological monoid.  We write $BG:=B(\ast,G,\ast)$ and $EG:=B(\ast,G,G)$, $\EG=B(G,G,\ast)$ where $B(-,-,-)$ denotes the geometric realization of the two sided bar construction in the monoidal category of unbased spaces with the Cartesian product.  We also consider the space $B(G,G,G)$ and morphisms $G\stackrel{\mu}\leftarrow B(G,G,G)\to EG \to BG$ between them.
Here $\mu$ is the map induced by multiplication $G\times G^n\times G \to G$, and the other maps are induced by the $G$ bimodule map $G \to \ast$. 

A point in $BG$ (resp.\ $EG$) may be represented by a diagram on the left (resp.\ right):
\newline
\begin{center}
	\includegraphics[scale=1.5]{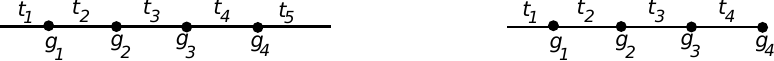}
\end{center}
where $g_i\in G$, $0\leq t_i \leq 1$ and $\sum t_i=1$.  A point is not represented uniquely by such a diagram, $BG$ has identifications generated by:
\begin{center}
\includegraphics[scale=.45]{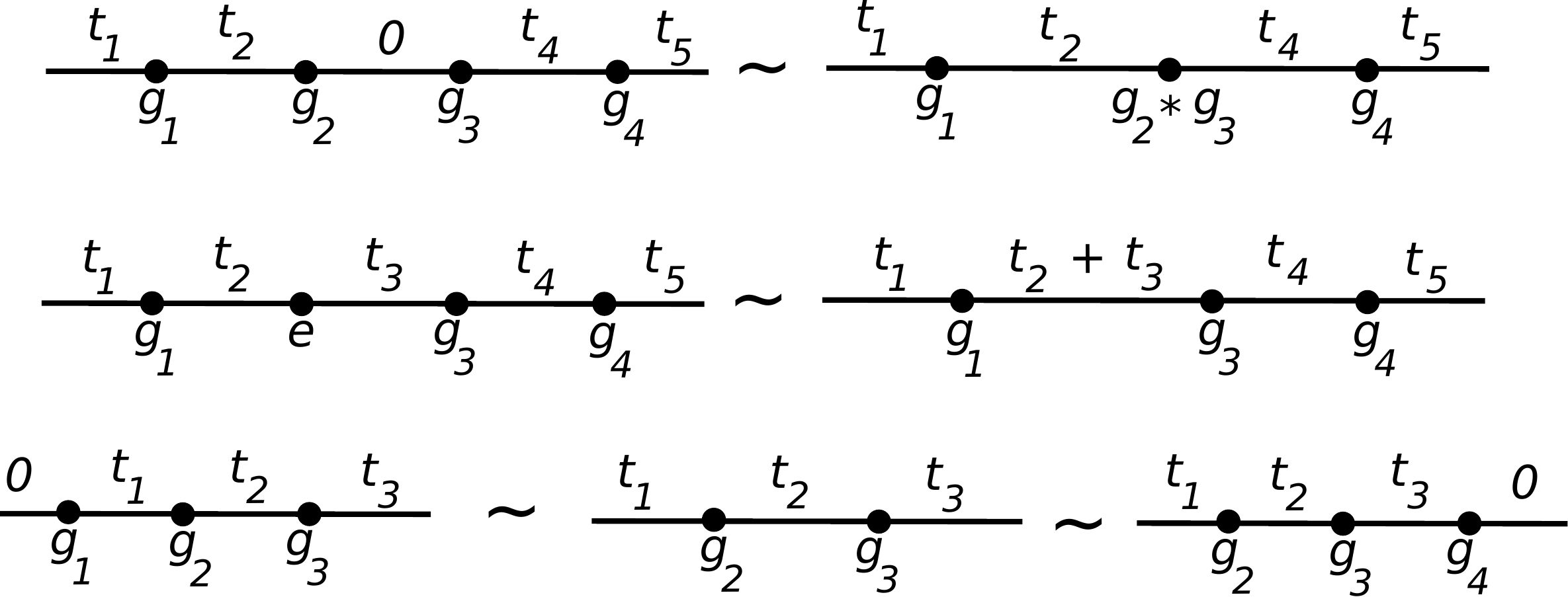}
\end{center}
Here $e\in G$ is the unit and we observe that this description is unambiguous for $e$ on an end.  Similarly for $EG$ except on the right hand side which identifies:

\begin{center}
	\includegraphics[scale=1.5]{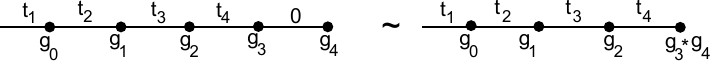}
\end{center}

The space $EG$ (resp.\ $\EG$) is a right (resp.\ left) $G$-space with $G$ action given by multiplying the right (resp.\ left)  end point on the right (resp.\ left) hand side.  The diagrammatic description makes it clear that this is a well defined and free $G$-action.  The map $EG\to BG$ (resp. $\EG\to BG$) is depicted by removing the end point which we observe diagrammatically to be well defined. 


A homotopy which contracts $EG$, call it $h_e$, can be defined as $h_e(s,-)$ grafting an $e$ labeled vertex with weight $s$ and scaling all other weights by $(1-s)$.  For example $h_e(0,p)$ is on the left then $h_1(s,p)$ is center and $h_1(1,p)$ is on the right:
\newline
\begin{center}
	\includegraphics[scale=1.5]{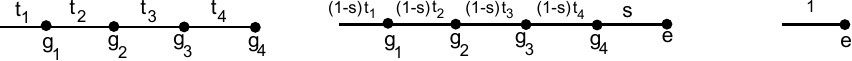}		
\end{center}
Note we could contract to any point in $EG$ by such a grafting and scaling.

Finally, a point in the space $B(G,G,G)$ may be represented non-uniquely by a diagram:
\begin{center}
	\includegraphics[scale=1.5]{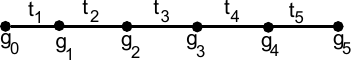}
\end{center}
where again $g_i,t_i$ and equivalence relations as defined above.  It can be seen diagrammatically that multiplying the vertex labels in the prescribed order is a well defined operation.  Grafting and scaling as above gives a homotopy between $B(G,G,G)$ and $G$.

\subsection*{Acknowledgment}  

I would like to thank Gabriel Drummond-Cole for several helpful discussions about this result. The treatment of cooperad structures in Section $\ref{cooperad}$ and in particular Proposition $\ref{coopmap}$ were aided significantly by his input.  I would also like to thank Greg Arone for his help and encouragement.  I would like to gratefully acknowledge the support of both Stockholm University and IHES where this paper was written.  Finally I would like to thank an anonymous referee for a number of helpful suggestions and comments.


\end{document}